\documentclass[reqno]{amsart}
\usepackage[left=1in,right=1in,top=1in,bottom=1in]{geometry}
\usepackage{tikz}
\usetikzlibrary{shapes,decorations,calc,arrows}
\usepackage[foot]{amsaddr}
\usepackage{amsthm}
\usepackage{amscd}
\usepackage{amsfonts}
\usepackage{amsmath}
\usepackage{amssymb}
\usepackage{mathrsfs}
\usepackage{multirow}
\usepackage{verbatim}
\usepackage{url}
\usepackage[hidelinks]{hyperref}
\usepackage{graphicx}
\usepackage{cite}
\usepackage{fancyhdr}
\usepackage{yfonts}
\usepackage{setspace}
\usepackage{titlesec}
\usepackage{enumitem}

\usepackage{ dsfont }

\titleformat{\section}[hang]
{\normalfont\Large\bfseries}
{\thesection.}{0.5em}{}

\titlespacing*{\section}{0pc}{2pc}{0.25pc}

\titleformat{\subsection}[runin]
{\normalfont\large\bfseries}
{\thesubsection}{0.5em}{}

\titlespacing{\subsection}{0pc}{1.5pc}{0.5pc}

\newcommand{\supp}{\text{supp}}

\newcommand{\N}{\mathbb{N}}

\newcommand{\R}{\mathbb{R}}
\newcommand{\C}{\mathbb{C}}

\newcommand{\B}{\mathcal{B}}

\newcommand{\J}{\mathscr{J}}

\newcommand{\HS}{\text{HS}}
\newcommand{\1}{\mathds{1}}

\newcommand{\abs}[1]{\left|#1\right|}
\newcommand{\norm}[1]{\left\|#1\right\|}
\newcommand{\set}[1]{\left\{#1\right\}}
\newcommand{\paren}[1]{\left(#1\right)}
\newcommand{\ang}[1]{\left\langle#1\right\rangle}

\newcommand{\BT}{B\ang{T}}
\newcommand{\BX}{B\ang{X}}

\DeclareMathOperator{\ev}{ev}

\newcommand{\II}{\rm{II}}

\newcommand{\<}{\left\langle}
\renewcommand{\>}{\right\rangle}
\renewcommand{\Re}[1]{\text{Re}\ #1}

\newcommand{\dom}{\text{dom}}

\newcommand{\mc}[1]{\mathcal{#1}}

\newcommand{\mmop}{{M\bar\otimes M^\circ}}

\newtheorem{thm}{Theorem}[section]
\newtheorem{prop}[thm]{Proposition}
\newtheorem{lem}[thm]{Lemma}
\newtheorem{cor}[thm]{Corollary}

\theoremstyle{definition}
\newtheorem{defi}[thm]{Definition}
\newtheorem{ex}[thm]{Example}

\newtheorem{rem}[thm]{Remark}

\title{\textbf{Free Stein Irregularity and Dimension}}
\author{Ian Charlesworth$^\circ$}
\address{$^\circ$Department of Mathematics, University of California, Berkeley \hfill \url{ilc@math.berkeley.edu}}
\author{Brent Nelson$^\bullet$}
\address{$^\bullet$Department of Mathematics, Michigan State University \hfill \url{brent@math.msu.edu}}
\date{}

\begin{document}

\maketitle

\begin{abstract}
		We introduce a free probabilistic quantity called free Stein irregularity, which is defined in terms of free Stein discrepancies.
		It turns out that this quantity is related via a simple formula to the Murray--von Neumann dimension of the closure of the domain of the adjoint of the non-commutative Jacobian associated to Voiculescu's free difference quotients.
		We call this dimension the free Stein dimension, and show that it is a $*$-algebra invariant.
		We relate these quantities to the free Fisher information, the non-microstates free entropy, and the non-microstates free entropy dimension.
		In the one-variable case, we show that the free Stein dimension agrees with the free entropy dimension, and in the multivariable case compute it in a number of examples.
\end{abstract}


\section*{Introduction.}

In free probability, given an $n$-tuple of self-adjoint operators $X:=(x_1,\ldots, x_n)$ in a tracial von Neumann algebra $(M,\tau)$, a \emph{regularity condition} is some quantitative behavior of the joint distribution of $X$ that implies some qualitative behavior of the individual operators $x_1,\ldots, x_n$ or the algebras (von Neumann or otherwise) that they generate.
All of the well-studied regularity conditions fall broadly into two categories: microstates and non-microstates.
Examples of the former include Voiculescu's microstates free entropy $\chi(X)$, microstates free entropy dimension $\delta(X)$ \cite{VoiII}, modified microstates free entropy dimension $\delta_0(X)$ \cite{VoiIII}, upper free orbit dimension $\mathfrak{K}_2(X)$ \cite{HS07}, and 1-bounded entropy $h(W^*(X))$ \cite{Hay18}.
Examples of the latter include non-microstates free entropy $\chi^*(X)$, free Fisher information $\Phi^*(X)$ \cite{VoiV}, non-microstates free entropy dimensions $\delta^*(X)$ and $\delta^*(X)$, and $\Delta(X)$ \cite{CS05}.

Roughly speaking, microstates quantities examine the joint distribution of $X$ in terms of how well it is approximated by finite dimensional matrix algebras, whereas non-microstates quantities consider the behavior of certain derivations on the polynomial algebra generated by $x_1,\ldots, x_n$. We recall a few of the regularity conditions corresponding to the aforementioned free probabilistic quantities:
	\begin{itemize}
		\item If $\Phi^*(x)<\infty$, then the spectral measure of $x$ is Lebesgue absolutely continuous with density in $L^3(\R,m)$ \cite{VoiI}.

		\item If $\delta(x)=1$, then $x$ is diffuse (i.e. its spectral measures has no atoms) \cite{VoiII}.

		\item If $\delta_0(X)>1$, then $W^*(X)$ has no Cartan subalgebras and does not have property $\Gamma$ \cite{VoiIII}.

		\item If $\delta_0(X)>1$, then $W^*(X)$ is prime \cite{Ge98}.

		\item If $\Phi^*(X)<\infty$, then $W^*(X)$ does not have property $\Gamma$ \cite{Dab10}.

		\item If $\delta^*(X)=n>1$, then $W^*(X)$ is a factor \cite{Dab10}.

		\item If $\delta^*(X)=n$, then every non-constant, self-adjoint $p\in \C\<x_1,\ldots, x_n\>$ is diffuse \cite{CS16, MSW17}.

		\item If $\Phi^*(X)<\infty$, then $\chi^*(p)>-\infty$ for every non-constant, self-adjoint $p\in \C\<x_1,\ldots, x_n\>$ \cite{BM18}.
	\end{itemize}

	In the present paper, we propose new quantities that fall into the non-microstates category: \emph{free Stein irregularity} and \emph{free Stein dimension}  (see Definitions \ref{def:free_Stein_irregularity} and \ref{def:free_Stein_dim}).
Motivated by work of the second author in \cite{FN17}, these quantities are defined via the free analogues of Stein kernels and Stein discrepancy (see \cite{LNP15} and its references). Given an $n$-tuple $(\xi_1,\ldots, \xi_n)\in L^2(M)^n$, the free Stein discrepancy of $X$ relative to this $n$-tuple (see Subsection \ref{subsec:free_Stein_kernels_discrepancy}) is a non-negative quantity that measures how close $\xi_1,\ldots, \xi_n$ are to being the conjugate variables to $x_1,\ldots, x_n$. In particular, the free Stein discrepancy is zero if and only if $\xi_1,\ldots, \xi_n$ are the conjugate variables, in which case $\Phi^*(X)<\infty$ and so the above results tell us that $W^*(X)$ does not have property $\Gamma$ and $\chi^*(p)>-\infty$ for every non-constant, self-adjoint $p\in \C\<x_1,\ldots, x_n\>$. Of course, determining that the free Stein discrepancy was zero required preexisting knowledge of the $n$-tuple $(\xi_1,\ldots, \xi_n)$ --- or a very lucky guess. 

In this paper, we explore what can be said if instead one merely supposes that the free Stein discrepancy can be made arbitrarily small by varying the $n$-tuple $(\xi_1,\ldots, \xi_n)\in L^2(M)^n$. We are therefore naturally driven to consider the infimum of free Stein discrepancies, which we define as the free Stein irregularity, and the situation of interest is simply the regularity condition of having zero free Stein irregularity. One immediately has that this is a weaker regularity condition than $\Phi^*(X)<\infty$, but it turns out to be a stronger condition than $\delta^*(X)=n$ (see Corollary \ref{cor:free_Stein_info_and_free_entropy_dim}). Interestingly, in the one variable case $\delta^*(X)=1$ is equivalent to having zero free Stein irregularity. This is because for $X=(x_1)$, the square of the free Stein irregularity can be computed explicitly and is given by the sum of the squares of masses of any atoms in the spectral measure of $x_1$ (see Theorem~\ref{thm:one_variable_computation}). In the general case, the free Stein irregularity is (somewhat surprisingly) given by a formula involving the Murray--von Neumann dimension of the domain of an unbounded operator (see Theorem \ref{thm:free_Stein_info_is_dim}): namely, the adjoint of the non-commutative Jacobian associated to Voiculescu's free difference quotients (see Subsection \ref{subsec:notation}). We call this dimension the free Stein dimension of $X$, and are able to further relate it to a module of closable derivations on the $\C\<x_1,\ldots, x_n\>$. From this characterization it follows that the free Stein dimension is a $*$-algebra invariant (see Theorem \ref{thm:algebra_inv}). Furthermore, we also consider the above quantities when $x_1,\ldots, x_n$ are considered as variables over a unital $*$-subalgebra $B\subset M$.

The structure of the paper is as follows. In Section~\ref{sec:prelims} we establish some notation and recall the definitions of free Stein kernels and free Stein discrepancy. In Section~\ref{sec:freesteinirreg}, we define free Stein irregularity, derive some elementary properties, and define free Stein dimension. In Section~\ref{sec:derivations}, we characterize free Stein dimension through modules of closable derivations and use this to show algebraic invariance. In Section~\ref{sec:relation_to_free_entropy}, we relate the free Stein irregularity and dimension to free Fisher information and non-microstates free entropy dimension(s), and compute the both explicitly in the one-variable case. In Section~\ref{sec:computations}, we compute the (multivariable) free Stein irregularity and dimension for a tuple of generating a group algebra or finite-dimensional algebra. We conclude the paper with a few appendices detailing interesting examples and computations.

\subsection*{Acknowledgments.}
The authors would like to thank Dimitri Shlyakhtenko for his useful comments and suggestions; in particular, for suggesting a cleaner approach to the results in Section~\ref{sec:derivations}.
They would also like to thank Michael Hartglass, Benjamin Hayes, and David Jekel for helpful discussions related to this paper.
This work was initiated while the authors were attending the Park City Mathematics Institute (PCMI) Summer Session on Random Matrices.
Part of this research was performed while the authors were visiting the Institute for Pure and Applied Mathematics (IPAM), which is supported by the National Science Foundation.
The first and second authors were supported by NSF grants DMS-1803557 and DMS-1502822, respectively.


\section{Preliminaries.}
\label{sec:prelims}

\subsection{Notation.}\label{subsec:notation}
Throughout $(M,\tau)$ denotes a tracial $W^*$-probability space. We denote by $L^2(M)$ the GNS Hilbert space corresponding to $\tau$ and identify $M$ with its representation on this space. We let $M^\circ=\{x^\circ\colon x\in M\}$ denote the opposite von Neumann algebra, represented on $L^2(M^\circ)$ which can be identified with the dual Hilbert space to $L^2(M)$. We let $\mmop$ denote the von Neumann algebra tensor product, which is equipped with the tensor product trace $\tau\otimes\tau^\circ$. We will typically repress the `$\circ$' notation on elements of $\mmop$.

Throughout, $X$ will denote a tuple $(x_1, \ldots, x_n) \in M^n$ (not necessarily self-adjoint), and $B \subset M$ will be a unital subalgebra.
We will always assume that for each $i=1,\ldots,n$, $x_i^*=x_j$ for some $j$ (possibly $j=i$ if $x_i$ is actually self-adjoint).
$T$ will denote a family $(t_1, \ldots, t_n)$ of indeterminates of the same length as $X$, and $\BT$ will be the algebra generated by $t_1, \ldots, t_n$ and $B$.
Note that there is a unique unital homomorphism $\BT \to \BX$ which sends each $t_i$ to $x_i$, which we will denote $\ev_X$; $\ev_X$ is always surjective but may fail to be injective. We will also use $\ev_X$ to denote the corresponding maps on $\BT^n$, $\BT\otimes\BT^\circ$, and $M_n(\BT\otimes \BT^\circ)$.

For each $i$, the \emph{free difference quotient} $\partial_i : \BT \to \BT\otimes \BT^\circ$ is defined to be the (unique) linear map with $\partial_i(t_j) = \delta_{i=j}$ and $B \subset \ker\partial_i$, satisfying the Leibniz rule; more precisely,
\[
	\partial_i(b_0t_{i_1}b_1t_{i_2}\cdots b_{d-1}t_db_{d})
	= \sum_{k=1}^d \delta_{i_k=i} b_0t_{i_1}b_1\cdots t_{k-1}b_{k-1}\otimes b_kt_{k+1}\cdots t_nb_n.
\]
We similarly define $\partial : \BT \to (\BT\otimes \BT^\circ)^n$ by $\partial p = (\partial_1p, \ldots, \partial_np)$, and let $\J : \BT^n \to M_n(\BT\otimes \BT^\circ)$ be the non-commutative Jacobian:
	\[
		\J(p_1,\ldots, p_n) = \left[\begin{array}{ccc} \partial_1 p_1 &\cdots& \partial_n p_1 \\ \partial_1 p_2 & \cdots & \partial_n p_2 \\ \vdots & \ddots & \vdots \\ \partial_1 p_n & \cdots & \partial_n p_n \end{array}\right] = \left[\begin{array}{c} \partial p_1 \\ \partial p_2 \\ \vdots \\ \partial p_n \end{array}\right].
	\]

For $\Xi=(\xi_1,\ldots, \xi_n),H=(\eta_1,\ldots, \eta_n)$ in either $L^2(M)^n$ or $L^2(M\bar\otimes M^{\circ})^n$ denote
	\[
		\<\Xi, H\>_2:= \sum_{j=1}^n \<\xi_j,\eta_j\>_2.
	\]
For $A,B\in M_n(L^2(\mmop))$ denote
	\[
		\<A, B\>_{\HS}:= \sum_{j,k=1}^n \< [A]_{jk}, [B]_{jk}\>_2.
	\]
We denote by $\#$ the usual product in $\mmop$ ($(a\otimes b)\# (c\otimes d)=(ac)\otimes (db)$), the usual product in $M_n(\mmop)$, the action of $\mmop$ on $L^2(M)$ ($(a\otimes b)\# c = acb$), the diagonal action of $\mmop$ on $L^2(M)^n$, and the action of $M_n(\mmop)$ on $L^2(M)^n$.

	In the case that $X$ satisfy no $B$-algebraic relations, we can view $\partial_i, \partial,$ and $\J$ defined on polynomials in the variables $X$ rather than the indeterminates $T$, and so they become densely-defined operators on $L^2(\BX)$ or $L^2(\BX)^n$ with codomains $L^2(\mmop), L^2(\mmop)^n$, or $M_n(L^2(\mmop))$, respectively.

	We denote by $\partial_{i:B}^*$, $\partial_{X:B}^*$, and $\J_{X:B}^*$ the adjoints of the implied relations on $B\ang{X}$; for example, we define $\partial_{i\colon B}^*$ to be the map with domain consisting of those $a \in L^2(\mmop)$ for which there is some $\eta \in L^2(\BX)$ such that for all $p \in \BT$ we have
	\[\ang{\eta, \ev_X(p)} = \ang{a, \ev_X\circ\partial_i(p)};\]
	we then set $\partial_i^*(a) = \eta$.
Thus $\partial_{i:B}^* : L^2(\BX\otimes\BX^\circ) \to L^2(\BX)$, $\partial_{X:B}^* : L^2(\BX\otimes\BX^\circ)^n \to L^2(\BX)$, and $\J_{X:B}^* : M_n(L^2(\BX\otimes\BX^\circ)) \to L^2(\BX)^n$ are unbounded operators, although their domains may fail to be dense.
When $X$ or $B$ are clear from context, we may suppress the relevant subscript.

Lastly, let us denote
	\[
		\1:=\left[\begin{array}{ccc} 1\otimes 1 & & 0 \\  & \ddots & \\ 0 & & 1\otimes 1 \end{array}\right],
	\]
so that $\J(X)=\1$.

\subsection{Free Stein kernels and free Stein discrepancy.}\label{subsec:free_Stein_kernels_discrepancy}
We recall some definitions below from \cite{FN17}.
These have been modified slightly to accommodate our consideration of non-algebraically free operators $X$ over a unital subalgebra $B$, but when $X$ is algebraically free and $B=\C$, we recover the original definition.
By working in this broader generality, we reap a number of benefits: we are able to consider freeness with amalgamation; we are able to compute free Stein dimensions in finite-dimensional algebras; and we are able derive some interesting statements about the free Stein dimension of certain generators of interpolated free group factors (see Appendix~\ref{app:examples}).
The reader may find it useful to gain intuition by considering (as the authors have) the simpler case outlined in Remark~\ref{rem:scalarcase}, where $B = \C$, $X$ is algebraically free and self-adjoint, and the free difference quotients are densely defined operators.

Given $\Xi\in L^2(M)^n$, we say that
	\[
		A\in M_n( L^2(\BX\otimes \BX^\circ))
	\]
is a \emph{free Stein kernel} of $X$ relative to $\Xi$ over $B$ if $A \in \dom(\J_{X:B}^*)$ and $\J_{X:B}^*(A) = \Xi$: to wit, if
	\begin{align}
	\ang{ \Xi, \ev_X P }_2 = \ang{ A, \ev_X\circ\J (P) }_{\HS}\qquad \forall P\in \BT^n.
	\end{align}
In this case we say (after \cite{Shl04}) that $\Xi$ is a \emph{partial conjugate variable} to $X$ corresponding to $A$.

The \emph{free Stein discrepancy} of $X$ relative to $\Xi$ over $B$ is the quantity
\[
	\Sigma^*(X\mid \Xi : B) := \inf_A \| A - \1\|_{\HS},
\]
where as before the infimum is over all free Stein kernels of $X$ relative to $\Xi$ over $B$.
Equivalently, $\Sigma^*(X\mid \Xi : B) = \norm{\Pi(A) - \1}_{\HS}$ where $A$ is any free Stein kernel of $X$ relative to $\Xi$ and $\Pi$ is the orthogonal projection onto the closure of the range of $\ev_X\circ\J$.

A priori the free Stein discrepancy could be infinite, since a free Stein kernel for $X$ need not exist. Indeed, if $\Xi$ is not orthogonal to $B^n \subset L^2(M)^n$ then for some $Z\in B^n$ we have
	\[
		\<\Xi, \ev_XZ\>_2 \neq 0 =\< A, \ev_X\circ \J (Z)\>_{\HS}\qquad \forall A\in M_n(L^2(M\bar\otimes M^{\circ})).
	 \]
For general unital subalgebras $B$, it is not clear if the condition $\Xi\perp B^n$ is sufficient to guarantee the existence of free Stein kernels.
However, in the case $B=\C$ it suffices by \cite[Theorem 2.1]{CFM18}, which we state below.

\begin{prop}[\!\!\cite{CFM18}]\label{prop:Mai}
For $\Xi=(\xi_1,\ldots, \xi_n)\in L^2(M)^n\ominus \C^n$,
	\[
		A_\Xi:=\left[\frac12 (\xi_i\otimes 1 - 1\otimes \xi_i)\#(x_j\otimes 1 - 1\otimes x_j) \right]_{i,j=1}^n\in M_n(L^2(\mmop))
	\]
is a free Stein kernel for $X$ relative to $\Xi$. Consequently, $\Sigma^*(X\mid \Xi)<\infty$ always.
\end{prop}

\begin{rem}
	For larger unital subalgebras $B$, $A_\Xi$ given in Proposition~\ref{prop:Mai} may fail to be a free Stein kernel. Indeed, if $x, s$ are freely independent semicircular variables, $B = \C[s]$, and $\xi = sxs$, one can compute that $\ang{\xi, sxs} = 1$ while
	$$\ang{\frac12(\xi\otimes1-1\otimes\xi)\#(x\otimes1-1\otimes x), s\otimes s} = 0.$$
\end{rem}

One might hope that in nice cases $A_\Xi$ is the free Stein kernel which attains the free Stein discrepancy of $X$, but unfortunately this holds if and only if $\Xi = 0$ (see Appendix~\ref{appendix:maikernel}).
However, we do obtain the following corollary:

\begin{cor}
The map
	\[
		L^2(M)^n\ominus \C^n\ni\Xi\mapsto \Sigma^*(X\mid \Xi)
	\]
is continuous.
\end{cor}
\begin{proof}
For $\Xi, \Xi'\in L^2(M)^n\ominus \C^n$ let $A_\Xi$ and $A_{\Xi'}$ be as in Proposition \ref{prop:Mai}. Then
	\begin{align*}
		| \Sigma^*(X\mid\Xi) - \Sigma^*(X\mid \Xi')| &=  \abs{ \|\Pi(A_\Xi) - \1\|_{\HS} - \|\Pi(A_{\Xi'}) - \1 \|_{\HS} }\\
			&\leq \| \Pi(A_\Xi) - \Pi(A_{\Xi'})\|_{\HS} \leq \|A_{\Xi} - A_{\Xi'}\|_{\HS} \leq C \|\Xi - \Xi'\|_2,
	\end{align*}
where $C>0$ is a constant depending only on $n$ and $X$.
\end{proof}

\begin{rem}
If $\Sigma^*(X\mid \Xi)=0$, then $\1$ is a free Stein kernel for $X$ and hence
	\[
		\<\Xi, \ev_X P\>_2 = \<\1, \ev_X\circ\J (P)\>_{\HS}\qquad \forall P\in \C\<X\>^n.
	\]
That is, $\Xi$ is the usual conjugate variable to $X$. In fact, this is precisely why the free Stein discrepancy is defined to measure the distance between a free Stein kernel $A$ and $\1$. We remind the reader that the \emph{free Fisher information} of $X$ is defined as the quantity
	\[
		\Phi^*(X):= \|\Xi\|_2^2
	\]
if $\Xi$ is the conjugate variable to $X$, whereas it is defined to be $+\infty$ if no conjugate variable exists (cf. \cite[Definition 6.1]{VoiV}).

Furthermore, $\Sigma^*(X\mid X)=0$ if and only if $X$ is the conjugate variable to $X$ if and only if $X$ is a free semicircular family.
\end{rem}


\section{Free Stein Irregularity.}
\label{sec:freesteinirreg}

We begin with the definition of free Stein irregularity. In order to better motivate and clarify the definition, it is followed by an examination of a special case.

\begin{defi}
	\label{def:free_Stein_irregularity}
	Let $X=(x_1,\ldots, x_n)\in M^n$ be a tuple of operators such that for each $i=1,\ldots, n$, $x_i^*=x_j$ for some $j$.
	Let $B$ be a unital $*$-subalgebra of $M$.
	The \textbf{free Stein irregularity} of $X$ over $B$ is the quantity
	\[
		\Sigma^*(X:B) := \inf\set{\Sigma^*(X \mid \Xi : B) : \Xi \in L^2(M)^n}.
	\]
For $R>0$, the \textbf{$R$-bounded free Stein irregularity} of $X$ over $B$ is the quantity
	\[
		\Sigma_R^*(X\colon B) := \inf\set{\Sigma^*(X \mid \Xi : B) : \Xi \in L^2(M)^n\text{ with }\|\Xi\|_2\leq R}.
	\]
Note that $\displaystyle \Sigma^*(X\colon B) = \inf_{R>0} \Sigma_R^*(X\colon B) = \lim_{R\to\infty} \Sigma^*_R(X\colon B)$. In the particular case $B = \C$, we will use the shorthand $\Sigma^*(X) := \Sigma^*(X:\C)$.
\end{defi}

\begin{rem}\label{rem:relativekernel}
Notice that if $B \subseteq C \subseteq M$, there are fewer free Stein kernels of $X$ over $B$ than over $C$ (as there are more polynomials and so more relations must be satisfied); it follows that $\Sigma^*(X : B) \leq \Sigma^*(X : C)$. More formally, if $\mc{E}\colon \mmop \to W^*(\BX \otimes \BX^\circ)$ is the trace-preserving conditional expectation onto the von Neumann algebra generated by $\BX\otimes \BX^\circ$, then the claimed inequality follows from the inclusion $\left(\mc{E}\otimes I_n\right)(\dom(\J_{X\colon C}^*))\subset \dom(\J_{X\colon B}^*)$. 	
\end{rem}

\begin{rem}\label{rem:scalarcase}
Consider the following special case: let $X=(x_1,\ldots, x_n)$ be an $n$-tuple of self-adjoint operators generating $M$. Assume that $x_1,\ldots, x_n$ are algebraically free so that $\ev_X\colon \C\<T\>\to \C\<X\>$ is a $*$-algebra isomorphism. This allows us to view the free difference quotients $\partial_j$, $j=1,\ldots, n$, as defined directly on $\C\<X\>$, and---moreover---as densely defined (unbounded) operators of the form
	\[
		\partial_j\colon L^2(M)\to L^2(\mmop).
	\]
	Similarly, $\partial$ and $\J$ may be regarded as maps densely defined on the appropriate Hilbert spaces.

In this context, a free Stein kernel $A$ of $X$ relative to some $\Xi$ is simply an element of $\dom(\J^*)$ with $\J^*(A)=\Xi$. Consequently, the \emph{free Stein irregularity}, which is given by the formula
	\[
		\Sigma^*(X):=\inf \{\Sigma^*(X\mid \Xi)\colon \Xi\in L^2(M)^n\},
	\]
(see Definition~\ref{def:free_Stein_irregularity}), is equivalently the distance between $\1$ and (the closure of) $\dom(\J^*)$ in $M_n(L^2(\mmop))$. The free Stein irregularity can be thought of as quantitative measurement of how close the $n$-tuple $X$ is to having conjugate variables. Indeed, capturing such a defect was the original motivation for defining this quantity and if we consider the following technical modification
	\[
		\Sigma_R^*(X):=\inf \{\Sigma^*(X\mid \Xi)\colon \Xi\in L^2(M)^n\text{ with }\|\Xi\|_2\leq R\},\qquad R>0
	\]
then $\Sigma_R^*(X)=0$ if and only if an $n$-tuple of conjugate variables to $X$ exists and is bounded by $R$ (see Theorem~\ref{thm:fisher_info_and_R-bounded_free_Stein_info}).

It turns out that the Hilbert subspace $\overline{\dom(\J^*)}\subset M_n(L^2(\mmop))$ is a left $\mmop$-module (see Lemma~\ref{lem:action_on_domain_of_adjoint}) and that its Murray--von Neumann dimension is related to the free Stein irregularity by the following formula:
	\[
		n- \Sigma^*(X)^2 = \dim_{\mmop}\left(\overline{\dom(\J^*)}\right)
	\]
(see Theorem~\ref{thm:free_Stein_info_is_dim}). We are thus compelled to study the quantity on the left-hand side, which we denote by $\sigma(X)$ and call the \emph{free Stein dimension} of $X$.
Analogously to free entropy dimension, it satisfies the inequality
	\[
		\sigma(X,Y)\leq \sigma(X)+\sigma(Y),
	\]
where $Y$ is another tuple of self-adjoint operators generating some (potentially larger) von Neumann algebra along with $X$; equality holds if $X$ and $Y$ are freely independent (see Corollary~\ref{cor:rel_additivity}).
It is also a $*$-algebra invariant (see Theorem~\ref{thm:algebra_inv}) and compares to the non-microstates free entropy dimensions:
	\[
		\sigma(X)\leq \delta^*(X) \leq \delta^*(X)
	\]
(see Corollary~\ref{cor:free_Stein_info_and_free_entropy_dim}). Moreover, it is known  to agree with these other dimensions in a number of cases (see Theorem~\ref{thm:one_variable_computation}, Proposition~\ref{prop:group_algebras}, and Corollary~\ref{cor:free_Stein_dimension_of_finite-dimensional_algebras}). In particular, when $n=1$ and $x$ is a self-adjoint operator with spectral measure $\mu_x$ we have
	\[
		\sigma(x) = 1 - \sum_{t\in \R} \mu_x(\{t\})^2.
	\]
It is thus natural to wonder whether these dimensions always agree. However, some basic relations still elude us. For example, when $\chi^*(X)>\infty$ it is known that $\delta^*(X)=\delta^*(X)=n$, but it remains open whether or not this implies $\sigma(X)=n$ as well. 
\end{rem}

\subsection{Elementary Properties.} 

We derive some useful properties of free Stein irregularity.

\begin{prop}
	$\Sigma^*(X:B) = \Sigma^*(X:W^*(B))$.
\end{prop}

\begin{proof}
	Denote $N = W^*(B)$.
	We have $\Sigma^*(X:B) \leq \Sigma^*(X:N)$ by Remark~\ref{rem:relativekernel}, so we need only establish the other inequality.
	Now, let us suppose that $A$ is a Stein kernel for $X$ relative to $\Xi$ over $B$.
	Fix $p = b_0t_{i_1}b_1\cdots t_{i_d}b_d \in N\<T\>$ with $b_0, \ldots, b_d \in N$ and for $j = 0, \ldots, d$ take a sequence $(b_j(k))_{k\in\N} \subset B$ converging strongly to $b_j$ with norms uniformly bounded by $\norm{b_j}$ (such exists by Kaplansky's density theorem).
	Then if we let
	\[p(k) := b_0(k)t_{i_1}b_1(k)\cdots t_{i_d}b_d(k) \in \BT,\]
	we find $\ev_X(p(k))$ converges to $\ev_X(p)$ in $L^2(M)$, while $\ev_X\circ\partial(p(k))$ converges to $\ev_X\circ\partial(p)$ in $L^2(\mmop)^n$.
	If $P \in N\<T\>^n$ and $P(k)$ are chosen in a similar way, it follows that
	\[
		\ang{\Xi, \ev_X P}
		= \lim_{k\to\infty}\ang{\Xi, \ev_X(P(k))}
		= \lim_{k\to\infty}\ang{A, \ev_X\circ\J(P(k))}
		= \ang{A, \ev_X\circ\J(P)};
	\]
	that is, $A$ is also a Stein kernel for $X$ relative to $\Xi$ over $N$.
	Hence $\Sigma^*(X:N) \leq \Sigma^*(X:B)$.
\end{proof}

\begin{lem}
	\label{lem:rel_free_inflation_v2}
	Let $B, C \subseteq M$ be unital $*$-subalgebras, and let $D \subseteq B \cap C$ be a common unital $*$-subalgebra with conditional expectation $\mc{E} : (B \vee C)\<X\> \to D$, where $B\vee C$ is the $*$-algebra generated by $B$ and $C$.
	If $C$ is free from $\BX$ with amalgamation over $D$, then
	\[
		\Sigma^*(X\colon B\vee C) = \Sigma^*(X\colon B).
	\]
	In particular, if $X$ is free from $C$, then $\Sigma^*(X\colon C) = \Sigma^*(X)$.
\end{lem}

\begin{proof}
	By Remark~\ref{rem:relativekernel}, it suffices to prove $\Sigma^*(X\colon B\vee C)\leq \Sigma^*(X\colon B)$. We will prove this by showing $\dom(\partial_{X:B}^*) \subset \dom(\partial^*_{X:B\vee C})$.
	Let $\eta \in \dom(\partial_{X:B}^*)$, with $\xi := \partial_{X:B}^*(\eta)$.
	Take $c_0, \ldots, c_d \in C$ with $\mc{E}(c_i) = 0$ for $i = 1, \ldots, d-1$, and $P_1, \ldots, P_d \in \BT$ with $\mc{E}(\ev_X P_i) = 0$ for $i = 1, \ldots, d$; set $P = c_0P_1c_1\cdots P_dc_d$.

	We claim
	\[
		\ang{\xi, \ev_X P}_2 = \ang{\eta, \ev_X\circ\partial_{X: B\vee C}P}_2.
	\]
	If $d = 0$, the left-hand side is $\ang{\xi, \mc{E}(c_0)} = 0$ since $\xi$ is orthogonal to $D \subseteq B$ and free from $C$ with amalgamation.
	The right-hand side is zero by the definition of $\partial_{X:B\vee C}$.
	If $d \geq 2$, it is not hard to check that both sides are zero due to freeness with amalgamation over $D$.
	Thus it remains to establish the claim when $d = 1$.
	In this case, invoking freeness with amalgamation, we have:
	\begin{align*}
		\ang{\xi, \ev_X(c_0 P_1 c_1)}_2
		&= \ang{\xi, c_0 (\ev_X P_1) c_1}_2 \\
		&= \ang{ \xi, \mc{E}(c_0) (\ev_X P_1) \mc{E}(c_1)}_2\\
		&= \ang{ \xi, \ev_X(\mc{E}(c_0) P_1 \mc{E}(c_1))}_2\\
		&= \ang{\eta, \ev_X\circ\partial_{X\colon B}(\mc{E}(c_0) P_1 \mc{E}(c_1))}_2 \\
		&= \ang{ \eta, \mc{E}(c_0)\cdot \ev_X\circ\partial_{X\colon B}(P_1)\cdot \mc{E}(c_1)}_2\\
		&= \ang{ \eta, \mc{E}(c_0)\cdot \ev_X\circ\partial_{X\colon B\vee C}(P_1)\cdot \mc{E}(c_1)}_2\\
		&= \ang{ \eta, c_0 \cdot \ev_X\circ\partial_{X\colon B\vee C}(P_1)\cdot c_1}_2 \\
		&= \ang{ \eta, \ev_X\circ\partial_{X\colon B\vee C}(c_0 P_1 c_1)}_2.
	\end{align*}
	This completes the proof of the claim.

	Finally, since such elements $P$ span $(B\vee C)\<T\>$, this shows that $\eta \in \dom(\partial^*_{X:B\vee C})$, completing the proof.
\end{proof}

\begin{thm}\label{thm:rel_subadditivity}
	Let $B, C \subset M$ be unital $*$-subalgebras, and let $X \in M^n$, $Y\in M^m$ be tuples.
	Then
	\[
		 \Sigma^*(X\colon B)^2 + \Sigma^*(Y\colon C)^2 \leq \Sigma^*(X,Y\colon B\vee C)^2 \leq  \Sigma^*(X\colon (B\vee C)\<Y\>)^2 + \Sigma^*(Y\colon (B\vee C)\<X\>)^2.
	\]
	Moreover, suppose $D\subset B\cap C$ is a common unital $*$-subalgebra with conditional expectation $\mc{E}\colon (B\vee C)[X\cup Y]\to D$. If $\BX$ and $C\<Y\>$ are free with amalgamation over $D$, then the above inequalities are equalities.
\end{thm}

\begin{proof}
	Let $A\in \dom(\J^*_{(X,Y)\colon B\vee C})$.
	Let $A_1$ be the entry-wise projection of the top-left $n\times n$ sub-matrix of $A$ onto $L^2(\BX\otimes \BX^\circ)$, and let $A_2$ be the entry-wise projection of the bottom-right $m\times m$ sub-matrix of $A$ onto $L^2(C\<Y\>\otimes C\<Y\>^\circ)$.
	One easily checks that $A_1\in \dom(\J_{X\colon B}^*)$ and $A_2\in \dom(\J_{Y\colon C}^*)$.
	Hence
	\[
		\Sigma^*(X\colon B)^2 + \Sigma^*(Y\colon C)^2 \leq \|A_1 - \1 \|_{\HS}^2  + \| A_2 - \1 \|_{\HS}^2 \leq \|A - \1 \|_{\HS}^2.
	\]
	Since $A\in \dom(\J_{(X,Y)\colon B\vee C}^*)$ was arbitrary, this yields the first inequality.

	Next, let $A_1\in \dom(\J_{X\colon (B\vee C)\<Y\>}^*)$ and $A_2\in \dom(\J_{Y\colon (B\vee C)\<X\>}^*)$.
	It is easily checked that
	\[
		A = \left[ \begin{array}{cc} A_1 & 0 \\ 0 & A_2 \end{array}\right] \in \dom(\J_{(X,Y)\colon B\vee C}^*).
	\]
	Thus
	\[
		\Sigma^*(X,Y\colon B\vee C)^2 \leq \| A - \1 \|_{\HS}^2 = \|A_1 - \1 \|_{\HS}^2 + \|A_2 - \1 \|_{\HS}^2,
	\]
	and so the second inequality follows. 

	Finally, if $\BX$ and $C\<Y\>$ are free with amalgamation over $D$, then by Lemma~\ref{lem:rel_free_inflation_v2} we have
		\[
			\Sigma^*(X\colon (B\vee C)\<Y\>) = \Sigma^*(X \colon B\vee (C\<Y\>)) = \Sigma^*(X\colon B).
		\]
	Similarly, $\Sigma^*(Y\colon (B\vee C)\<X\>) = \Sigma^*(Y\colon C)$.
	This forces the claimed equality.
\end{proof}

Applying the previous theorem to the special case $B=C=D(=\C)$, yields the following corollary.

\begin{cor}\label{cor:rel_additivity}
	\begin{enumerate}
		\item[]
		\item If $\BX$ and $B\<Y\>$ are free with amalgamation over B, then
			\[
				\Sigma^*(X,Y\colon B)^2 = \Sigma^*(X\colon B)^2 +\Sigma^*(Y\colon B)^2.
			\]

		\item If $X$ and $Y$ are free, then
			\[
				\Sigma^*(X,Y)^2 = \Sigma^*(X)^2+\Sigma^*(Y)^2.
			\]
	\end{enumerate}
\end{cor}

\begin{prop}
The function $R\mapsto \Sigma^*_R(X\colon B)$ is convex.
\end{prop}
\begin{proof}
Let $0<R_1<R_2$. Let $A_1,A_2\in \dom(\J_{X\colon B}^*)$ with $\|\J_{X\colon B}^*(A_i)\|_2 \leq R_i$, $i=1,2,$. Then for $t\in [0,1]$, $(1-t)A_1 + t A_2\in \dom(\J_{X\colon B}^*)$ with
	\[
		\| \J_{X\colon B}^*( (1-t)A_1+ tA_2)\|_2 \leq (1-t) \|\J_{X\colon B}^*(A_1)\|_2 + t \|\J_{X\colon B}^*(A_2)\|_2 \leq (1-t) R_1+ tR_2.
	\]
Hence
	\[
		\Sigma_{(1-t)R_1+tR_2}^*(X\colon B) \leq \| (1-t)A_1+ tA_2 - \1\|_{\HS} \leq (1-t) \|A_1 - \1\|_{\HS} + t\|A_2 - \1\|_{\HS}.
	\]
Taking the infimum over $A_1$ and $A_2$ completes the proof.
\end{proof}

\subsection{Free Stein dimension.}

In this subsection we give a characterization of the free Stein irregularity in terms of the Murray--von Neumann dimension of the closure of $\dom(\J_{X\colon B}^*)$ in $M_n(L^2(\mmop))$, viewed as a left $\mmop$-module.
We first show, in the following lemma, that $\dom(\partial_{X\colon B}^*)$ admits a left $\BX\otimes \BX^\circ$ action; this is the multivariate analogue of \cite[Proposition 4.1]{VoiV} and follows by an identical proof.

\begin{lem}\label{lem:action_on_domain_of_adjoint}
For $\eta=(\eta_1,\ldots, \eta_n)\in \dom(\partial_{X\colon B}^*)$ and $p,q\in \BT$, $\ev_X(p\otimes q)\# \eta \in \dom(\partial_{X\colon B}^*)$ with
	\begin{align*}
		\partial_{X\colon B}^*(\ev_X(p\otimes q)\# \eta) &=  \ev_X(p\otimes q) \# \partial_{X\colon B}^*(\eta)\\
		 &\qquad - \sum_{j=1}^n (1\otimes \tau^\circ)\left( \ev_Xp\cdot [\eta_j\#\ev_X\circ \partial_j(q^*)^*]\right) - (\tau\otimes 1)\left( [\eta_j\#\ev_X\circ\partial_j(p^*)^*]\cdot \ev_Xq \right).
	\end{align*}
\end{lem}

From this lemma we see that $\dom(\partial_{X\colon B}^*)$ is invariant under the left action of $\BX\otimes \BX^\circ$. Consequently, the Kaplansky density theorem implies that $\overline{\dom(\partial_{X\colon B}^*)}$ is a closed, left $\mmop$-module. Observe that for $A\in \dom(\J^*_{X\colon B})$, if $A_i=(A_{i1},\ldots, A_{in})$ (i.e. the $i$-th row of $A$) for $i=1,\ldots, n$, then $A_1,\ldots, A_n\in \dom(\partial_{X\colon B}^*)$. It then follows that $\overline{\dom(\J_{X\colon B}^*)}$ is also a closed, left $M_n(\mmop)$-module satisfying $\overline{\dom(\J_{X\colon B}^*)}\cong \overline{\dom(\partial_{X\colon B}^*)}^n$. This identification immediately gives the second equality in the following theorem.

\begin{thm}\label{thm:free_Stein_info_is_dim}
For $B\subset M$ a unital $*$-subalgebra and $X=(x_1,\ldots, x_n)\in M^n$  such that $M=W^*(\BX)$,
	\[
		n-\Sigma^*(X\colon B)^2 = \dim_{\mmop}\overline{\dom(\partial_{X\colon B}^*)}=\dim_{M_n(\mmop)}\overline{\dom(\J_{X\colon B}^*)}.
	\]
\end{thm}

\begin{proof}
Let $e\in M_n(L^2(\mmop))$ be the projection of $\1$ onto $\overline{\dom(\J_{X\colon B}^*)}$ so that $\Sigma^*(X\colon B)=\| e - \1\|_{\HS}$. Hence
	\[
		n - \Sigma^*(X\colon B)^2 = n - \|e\|_{\HS}^2 + 2 \Re\<e,\1\>_{\HS} - \|\1\|_{\HS}^2 = \|e\|_{\HS}^2.
	\]
Now, identify $\mmop$ with its diagonal representation $\mmop\otimes I_n$ on $L^2(\mmop)^n$.
Then $N:=(\mmop)'\cap \B(L^2(\mmop))$ is identified with $M_n(N)$.
Observe that 
	\begin{align*}
		\overline{\dom(\J_{X\colon B}^*)} &\cong \overline{\dom(\partial_{X\colon B}^*)}^n\\
		&\cong \overline{\left\{(Tv_1,\ldots, Tv_n)\in M_n(L^2(\mmop))\colon T\in M_n(N),\ T\cdot L^2(\mmop)^n \subset \overline{\dom(\partial_{X\colon B}^*)}\right\}},
	\end{align*}
where $v_j\in L^2(\mmop)^n$ is the vector with $1\otimes 1$ in the $j$-th entry and zeros elsewhere.
In fact, $(Tv_1,\ldots, Tv_n)$ in the last space is sent to its transpose in the first space.
Let $f$ be the projection of $L^2(\mmop)^n$ onto $\overline{\dom(\partial_{X\colon B}^*)}$, so that $f \in M_n(N)$; then $f^T\in \overline{\dom(\J_{X\colon B}^*)}$, and we further claim that $f^T=e$.
Indeed, for $A\in \dom(\J_{X\colon B}^*)$ let  $A_1,\ldots, A_n\in \dom(\partial_{X\colon B}^*)$ be the rows of $A$ as in the discussion preceding the theorem.
Hence $fA_i=A_i$ and so
	\begin{align*}
		\<f^T,A\>_{\HS} &= \< \1, f^T A\>_{\HS} = \sum_{i,j=1}^n \< 1\otimes 1, [f^T]_{ji} [A]_{ij}\>_{\HS}\\
		& = \sum_{i,j=1}^n \< 1\otimes 1, [f]_{ij} [A]_{ij}\>_{\HS} = \sum_{i,j=1}^n \< 1\otimes 1 , [f]_{ij} (A_i)_j \>_{\HS}\\
		&= \sum_{i=1}^n \<  1\otimes 1, (f A_i)_i\>_{\HS} = \sum_{i=1}^n \<  1\otimes 1, (A_i)_i\>_{\HS}\\
		&= \sum_{i=1}^n \< 1\otimes 1, [A]_{ii}\>_{\HS} = \<\1, A\>_{\HS} = \<e, A\>_{\HS}.
	\end{align*}
Thus $f^T = e$ and
	\[
		\dim_{\mmop}(\overline{\dom(\partial_{X\colon B}^*)}) = \|f\|_{\HS}^2 = \|f^T\|_{\HS}^2=\|e\|_{\HS}^2.
	\]
So the result follows by our previous computation.
\end{proof}

In light of the above theorem, we make the following definition.

\begin{defi}\label{def:free_Stein_dim}
For an $n$-tuple $X$, the \textbf{free Stein dimension} of $X$ over $B$ is the quantity
	\[
		\sigma(X\colon B):= n - \Sigma^*(X\colon B)^2.
	\]
\end{defi}

We can rephrase Theorem~\ref{thm:rel_subadditivity} and Corollary~\ref{cor:rel_additivity} in terms of free Stein dimension as follows:

\begin{cor}
	\label{cor:dim_additivity}
	Let $B, C \subset M$ be unital $*$-subalgebras, and let $X \in M^n$, $Y\in M^m$ be tuples.
	Then
		\[
			\sigma(X\colon (B\vee C)\<Y\>) + \sigma(Y\colon (B\vee C)\<X\>) \leq \sigma(X,Y\colon B\vee C) \leq  \sigma(X\colon B) + \sigma(Y\colon C).
		\]
	Moreover, suppose $D\subset B\cap C$ is a common unital $*$-subalgebra with conditional expectation $\mc{E}\colon (B\vee C)[X\cup Y]\to D$. If $\BX$ and $C\<Y\>$ are free with amalgamation over $D$, then the above inequalities are equalities.

	In particular, if $\BX$ and $B\<Y\>$ are free with amalgamation over B, then
			\[
				\sigma(X,Y\colon B) = \sigma(X\colon B) +\sigma(Y\colon B).
			\]
	Furthermore, if $X$ and $Y$ are free, then
			\[
				\sigma(X,Y) = \sigma(X)+\sigma(Y).
			\]
\end{cor}

\section{Via Closable Derivations.}
\label{sec:derivations}

In this section we characterize $\sigma(X\colon B)$ in terms of certain closable derivations on $\BX$.
This perspective yields a number of invariance results; in particular, that $\sigma(X\colon B)$ depends only on the algebras $B$ and $\BX$.

For an inclusion of two $*$-subalgebras $B\subset C\subset M$ with $M=W^*(C)$, consider the set
	\[
		\mathsf{Der}_{1\otimes 1}(B\subset C):=\{d\colon C\to L^2(\mmop)\mid d\text{ is a derivation with }B\subset\ker(d)\text{ and }1\otimes 1\in\dom(d^*)\}.
	\]
This set of derivations admits a right $C\otimes C^\circ$-action:
		\[
			d\cdot (a\otimes b):= d(\ \cdot\ )\#(a\otimes b).
		\]
Indeed, $a\otimes b\in \dom(d^*)$ by the same proof as \cite[Proposition 4.1]{VoiV} and so
	\[
		[d\cdot (a\otimes b)]^*(1\otimes 1) = d^*(a\otimes b).
	\]

\begin{lem}\label{lem:how_phi_X_works}
For $B\subset M$ a $*$-subalgebra and $X=(x_1,\ldots, x_n)\in M^n$, the conjugate linear map
	\begin{align*}
		\phi_X\colon \mathsf{Der}_{1\otimes 1}(B\subset \BX) &\to \dom(\partial_{X\colon B}^*)\\
			d&\mapsto \left( J_{\tau\otimes\tau^\circ}d(x_1),\ldots, J_{\tau\otimes\tau^\circ}d(x_n) \right)
	\end{align*}
is a bijection that maps the right $\BX\otimes \BX^\circ$-action on $\mathsf{Der}_{1\otimes 1}(B\subset \BX)$ to the left regular $\BX\otimes \BX^\circ$-action on $L^2(\mmop)$. Consequently, when $M=W^*(\BX)$
	\[
		\sigma(X\colon B) = \dim_{\mmop} \overline{\phi_X(\mathsf{Der}_{1\otimes 1}(B\subset \BX))}.
	\]
\end{lem}
\begin{proof}
First notice that each element of $\mathsf{Der}_{1\otimes 1}(B\subset \BX)$ is determined by its values on $X$. Hence $\phi_X$ is injective.

Now, given $d\in \mathsf{Der}_{1\otimes 1}(B\subset \BX)$, we have for any $p\in \BT$
	\[
		\<d^*(1\otimes 1), \ev_X p\>_2 = \< 1\otimes 1, d(\ev_X p)\> = \< 1\otimes 1, \sum_{i=1}^n \ev_X\circ \partial_i (p)\# d(x_i)\>_2 = \< \phi_X(d), \ev_X\circ\partial(p)\>_2.
	\]
Thus $\phi_X(d) \in \dom(\partial_{X\colon B}^*)$.

Given $a=(a_1,\ldots, a_n)\in \dom(\partial_{X\colon B}^*)$, define
	\[
		\tilde{d}_a\colon \BT\ni p\mapsto \sum_{i=1}^n \ev_X\circ \partial_i(p) \# J_{\tau\otimes\tau^\circ} a_i.
	\]
Then for $p,q,r\in \BT$ one has
	\[
		\<\ev_X(p\otimes q), \tilde{d}_a(r)\>_{\HS} = \< \ev_X(p\otimes q)\# a, \ev_X\circ\partial(r)\>_{\HS}= \< \partial_{X\colon B}^*(\ev_X(p\otimes q)\# a), \ev_X(r)\>_{\HS},
	\]
where the last equality uses Lemma~\ref{lem:action_on_domain_of_adjoint}. It follows that $\tilde{d}_a=d_a\circ \ev_X$ for some $d_a \in \mathsf{Der}_{1\otimes 1}(B\subset \BX)$ with $d_a^*(1\otimes 1) = \partial_{X\colon B}^*(a)$. In particular, $d_a(x_i)= \tilde{d}_a(t_i) = J_{\tau\otimes\tau^\circ}(a_i)$. Thus $a=\phi_X(d_a)\in \phi_X(\mathsf{Der}_{1\otimes 1}(B\subset \BX))$.
\end{proof}

\subsection{Algebraic invariance.}

If $Y\in \BX^m$ satisfies $B\<Y\>=\BX$, then $\phi_Y\circ\phi_X^{-1}$ yields a left $\BX\otimes \BX^\circ$-module isomorphism of $\dom(\partial_{X\colon B}^*)\cong \dom(\partial_{Y\colon B}^*)$. This extends to a left $\mmop$-module isomorphism $\overline{\dom(\partial_{X\colon B}^*)}\cong \overline{\dom(\partial_{Y\colon B}^*)}$. Using Theorem~\ref{thm:free_Stein_info_is_dim} we obtain the following theorem:

\begin{thm}\label{thm:algebra_inv}
If $Y\in \BX^m$ satisfies $B\<Y\>=\BX$, then
	\[
		\sigma(Y\colon B)=\sigma(X\colon B).
	\]	
\end{thm}

\begin{rem}
It follows from Theorem~\ref{thm:algebra_inv} that for any $Y\in \BX^m$, we have
	\[
		\sigma(X,Y\colon B)= \sigma(X\colon B).
	\]
In particular, if $Y\in B^m$ then $\sigma(Y\colon B)=0$.
\end{rem}

For every $Y\in \BX^m$ we have the following map:
	\begin{align*}
		\psi_{X,Y}\colon \mathsf{Der}_{1\otimes 1}(B\subset \BX) &\to \mathsf{Der}_{1\otimes 1}(B\subset B\<Y\>)\\
			d&\mapsto d\mid_{B\<Y\>}.
	\end{align*}
Of course, if $B\<Y\> = B\<X\>$ then this map is the identity map, but otherwise it is potentially neither injective nor surjective. Nevertheless, one can therefore always consider the composition $\phi_Y\circ \psi_{X,Y}\circ\phi_X^{-1}$.

\begin{prop}\label{prop:change_of_variables}
Let $Y\in \BX^n$ with $Y=\ev_X F$ for some $F\in \BT$. Then for $a\in \dom(\partial_{X\colon B}^*)$, we have
	\begin{align}\label{eqn:change_of_variables}
		\phi_Y\circ \psi_{X,Y}\circ\phi_X^{-1}(a) = a \# \ev_X\circ \J(F)^* \in \dom(\partial_{Y\colon B}^*),
	\end{align}
with $\partial_{Y\colon B}^*(a \# \ev_Y\circ \J(F)^*)=\partial_{X\colon B}^*(a)$.
Moreover, $\phi_Y\circ\psi_{X,Y}\circ\phi_X^{-1}$ extends to a map $\rho_{X,Y}\colon \overline{\dom(\partial_{X\colon B}^*)} \to \overline{\dom(\partial_{Y\colon B}^*)}$, and when $M=W^*(\BX)$ one has
	\[
		\sigma(X\colon B)\leq \sigma(Y\colon B) + \dim_{\mmop}(\ker(\rho_{X,Y})).
	\]
\end{prop}

\begin{proof}
Let $Y=(y_1,\ldots, y_m)$ and $F=(f_1,\ldots, f_m)$. For $a=(a_1,\ldots, a_n)\in \dom(\partial_{X\colon B}^*)$, we have that $\phi_X^{-1}(a)$ is given by the derivation $d_a$ defined in the proof of Lemma~\ref{lem:how_phi_X_works}. In particular, $d_a(x_i)= J_{\tau\otimes\tau^\circ} a_i$. It follows that
	\begin{align*}
		\phi_Y\circ\psi_{X,Y}\circ \phi_X^{-1}(a) &=( J_{\tau\otimes\tau^\circ} d_a(y_1),\ldots,J_{\tau\otimes\tau^\circ} d_a(y_m))\\
			&=\left( J_{\tau\otimes\tau^\circ} \sum_{i_1=1}^n  \ev_X\circ \partial_{i_1}(f_1)\# J_{\tau\otimes \tau^\circ} a_{i_1}, \ldots, J_{\tau\otimes\tau^\circ} \sum_{i_m=1}^n  \ev_X\circ \partial_{i_m}(f_m)\# J_{\tau\otimes \tau^\circ} a_{i_m} \right)\\
			&= \left( \sum_{i_1}^n a_{i_1} \# \ev_X\circ \partial_{i_1}(f_1)^*, \ldots, \sum_{i_m}^n a_{i_m} \#  \ev_X\circ \partial_{i_m}(f_m)^* \right)\\
			&= a \# \ev_X\circ \J(F)^*,
	\end{align*}
and so Equation~(\ref{eqn:change_of_variables}) holds. As the right action of $M_n(\BX\otimes\BX^\circ)$ is bounded, we immediately obtain the extension $\rho_{X,Y}$. Furthermore, $\rho_{X,Y}$ commutes with the left action of $\mmop$ and so is a left $\mmop$-module map when $M=W^*(\BX)$. Hence the claimed inequality follows from Theorem~\ref{thm:free_Stein_info_is_dim} and the rank--nullity theorem.
\end{proof}

	This structure in many cases puts restrictions on the sort of kernels that may be produced for a given tuple.
	For example, in light of Theorem~\ref{thm:rel_subadditivity} (and in particular its proof) one may ask whether a kernel for $X$ may always be extended to a kernel for a larger system $(X, Y)$, as in many nice cases this can be done.
	However, using the above proposition, Example~\ref{ex:no_extension} shows that this is not always possible.

\begin{rem}\label{rem:power_series_extension}
Theorem~\ref{thm:algebra_inv} and Proposition~\ref{prop:change_of_variables} can be generalized slightly by considering the following non-commutative power series. After \cite{CS16}, for $R>0$ we denote by $B\<T\>_R$ the completion of $\BT$ in the norm	
	\[
		\|p\|_R:= \inf\left\{ \sum \|b_0\|\cdots \|b_d\| R^d  \colon p=\sum b_0 t_{i_1}b_1\cdots t_{i_d}b_d,\ b_0,b_1,\ldots, b_d\in B\right\}.
	\]
	Note that this is in fact a Banach norm.
We also denote
	\[
		\BT_{>R} := \bigcup_{R'>R} \BT_{R'}.
	\]
	This space should be regarded as non-commutative power series with radius of convergence strictly greater than $R$. Observe that if $R\geq \max_i \|x_i\|$, the evaluation $\ev_X$ extends continuously to a homomorphism $\BT_{>R} \to M$ that sends $t_i$ to $x_i$. We denote $\BX_{>R}=\ev_X(\BT_{>R})$.

	It is readily seen that the derivations $\partial_i$, $i=1,\ldots, n$, extend to derivations on $\BT_{>R}$ that are valued in the projective tensor product $\BT_{>R}\hat\otimes \BT_{>R}^\circ$.
The evaluation map on $\BT\otimes \BT^{\circ}$ extends to $\BT_{>R}\hat\otimes \BT_{>R}^\circ$ and is valued in $\mmop$.
Consequently, when $R\geq \max_i \|x_i\|$, any $d\in \mathsf{Der}_{1\otimes 1}(B\subset \BX)$ can be extended to $\ev_X p\in \BX_{>R}$ by
		\[
			d(\ev_X p) := \sum_{i=1}^n \ev_X \partial_i(p)\# d(x_i).
		\]
That is,
	\[
		\mathsf{Der}_{1\otimes 1}(B\subset \BX) \subset \mathsf{Der}_{1\otimes 1}(B\subset \BX_{>R}).
	\]
In fact, the above inclusion is an equality. Indeed, all concerned derivations are closable by virtue of having $1\otimes 1$ in the domain of their adjoints. Consequently, such a derivation on $\BX_{>R}$ is uniquely determined by its values on $\BX$. It follows that for $Y\in \BX_{>R}^m$, if $B\<Y\>_{>R} = \BX_{>R}$ then $\sigma(Y\colon B)=\sigma(X\colon B)$. 
\end{rem}


\subsection{The special case of $B=\C$.}

We consider now the special case $B=\C$.
Of particular interest to us will be the case when $\partial$ gives a closable densely defined operator $\partial : L^2(M) \to L^2(\mmop)^n$, in which case we denote its closure by $\bar\partial$.
(We will see in Corollary~\ref{cor:full_free_Stein_dimension_iff_closable} that this is equivalent to the condition $\sigma(X) = n$.)
Since $\partial$ is a derivation which is \emph{symmetric} in the sense that
	\[
		\<a\cdot\partial(b),\partial(c)\>_2 = \<\partial(c^*),\partial(b^*)\cdot a^*\>_2\qquad a,b,c\in \C\<X\>,
	\] 
it follows from \cite{DL92} that $\bar\partial$ is a symmetric derivation on $\dom(\bar\partial)\cap M$, which is itself a $*$-algebra.

\begin{thm}\label{thm:extension_to_closure_of_domain}
	Let $M=W^*(X)$. Suppose $\partial: L^2(M) \to L^2(\mmop)^n$ gives a closable densely defined operator.
	Then for any $Y=(y_1,\ldots, y_m)\in \left(\dom(\bar\partial)\cap M\right)^m$ with $\bar\partial(y_j) \in (\mmop)^n$ for each $j=1,\ldots m$, we have $\sigma(X) = \sigma(X,Y)$.
\end{thm}
\begin{proof}
	First note that since $\dom(\bar\partial)\cap M$ is a $*$-algebra, it contains $\C\<X,Y\>$. Moreover, since each $\bar\partial y_j$ is a bounded operator, $\bar\partial p$ is a bounded operator for every $p\in \C\<X,Y\>$.

Now, given $d\in \mathsf{Der}_{1\otimes 1}(\C\subset\C\<X\>)$ define $\tilde{d}\colon \C\<X,Y\> \to L^2(\mmop)$ by
	\[
		\tilde{d}(p) = \sum_{i=1}^n \bar\partial_i(p)\# d(x_i).
	\]
We claim $\tilde{d}\in \mathsf{Der}_{1\otimes 1}(\C\subset \C\<X,Y\>)$. Indeed, it is a derivation by virtue of $\bar\partial$ being a derivation on $\dom(\bar\partial)\cap M\supset \C\<X,Y\>$. To see that $1\otimes 1\in \dom(\tilde{d}^*)$, note that for any $p\in \C\<X,Y\>$ there is a sequence $(p_k)_{k\in \N}\subset \C\<X\>$ converging to $p$ in $L^2(M)$ with $(\partial (p_k) )_{k\in \N}$ converging to $\bar\partial(p)$ in $L^2(\mmop)^n$. Consequently,
	\begin{align*}
		\< d^*(1\otimes 1), p\>_2 &= \lim_{k\to\infty} \< d^*(1\otimes 1), p_k\>_2\\
			&= \lim_{k\to\infty} \< 1\otimes 1, d(p_k)\>_2\\
			&= \lim_{k\to \infty} \sum_{i=1}^n \< \partial_i(p_k)^* , d(x_i)\>_2\\
			&= \sum_{i=1}^n \< \bar\partial_i(p)^*, d(x_i)\>_2 = \< 1\otimes 1, \tilde{d}(p)\>_2,
	\end{align*}
where the second-to-last equality follows from the fact that the adjoint is an isometry on $L^2(\mmop)$. Thus $1\otimes 1 \in \dom(\tilde{d}^*)$ with $\tilde{d}^*(1\otimes 1) = d^*(1\otimes 1)$. This establishes the claim.

Next consider $d\in \mathsf{Der}_{1\otimes 1}(\C\subset \C\<X,Y\>)$. We claim $d(y_j) = \bar\partial(y_j)\# d(X)$ for each $j=1,\ldots, m$. Indeed, for each $j=1,\ldots, n$ let $(y_j^{(k)})_{k\in\N}\subset \C\<X\>$ be a sequence converging to $y_j$ in $L^2(M)$ with $(\partial(y_j^{(k)}))_{k\in \N}$ converging to $\bar\partial(y_j)$ in $L^2(\mmop)^n$. Then for each $j=1,\ldots,m$ and any $a\in \C\<X\>\otimes \C\<X\>^\circ$ we have
	\begin{align*}
		\<d(y_j), a\>_2 &= \< y_j, d^*(a)\>_2 \\
			&= \lim_{k\to \infty} \< y_j^{(k)} , d^*(a)\>_2 \\
			&= \lim_{k\to\infty} \< d(y_j^{(k)}), a\>_2 \\
			&=\lim_{k\to\infty} \sum_{i=1}^n \< \partial_i(y_j^{(k)})\# d(x_i), a\>_2 \\
			&= \lim_{k\to\infty} \sum_{i=1}^n \< \partial_i(y_j^{(k)}), a\# J_{\tau\otimes\tau^\circ} d(x_i)\>_2\\
			&= \sum_{i=1}^n \< \bar\partial_i(y_j), a\# J_{\tau\otimes\tau^\circ} d(x_i)\>_2 = \< \bar\partial(y_j)\# d(X), a\>_2.
	\end{align*}
This yields the claimed equality since $\C\<X\>\otimes \C\<X\>^\circ$ is dense in $L^2(\mmop)$.

The first claim established the existence of a map
	\begin{align*}
		\mathsf{Der}_{1\otimes 1}(\C\subset \C\<X\>) &\to \mathsf{Der}_{1\otimes 1}(\C\subset \C\<X,Y\>)\\
			d&\mapsto \tilde{d}.
	\end{align*}
The second claim shows that every derivation in the latter set is completely determined by its values on the tuple $X$ and $\bar\partial(y_j)$ for $j=1,\ldots,m$. It follows that the above map is a bijection, and so by Lemma~\ref{lem:how_phi_X_works} we have $\sigma(X)=\sigma(X,Y)$.
\end{proof}

\begin{rem}
For $R\geq \max_i \|x_i\|$, Theorem~\ref{thm:extension_to_closure_of_domain} applies to any $y\in \C\<X\>_{>R}$ as in Remark~\ref{rem:power_series_extension}. It also applies to $f(p)$, where $p\in \dom(\bar\partial)\cap M$ is self-adjoint with $\bar\partial(p)\in(\mmop)^n$ and $f\in C^1(\R)$. In this case $\bar\partial(f(p))=\partial_p(f)\# \bar\partial(p)$, where $\partial_p(f)$ is the image of the function
	\[
		\tilde{f}(s,t) = \begin{cases} \frac{f(s) - f(t)}{s -t} & \text{if $s\neq t$} \\ f'(s) & \text{if s=t}\end{cases},
	\]
under the identification of the unital $C^*$-algebra generated by $p\otimes 1$ and $1\otimes p$ with continuous functions on its spectrum. Moreover, this can be further extended to Lipschitz functions $f$ on $\R$ (see \cite[Theorem 5.1]{DL92}).
\end{rem}

Lastly, we show that $1$ is a lower bound for $\sigma(X)$ as soon as $W^*(X)$ contains a diffuse element. In particular, this implies that $\sigma(X)\geq 1$ for any generating set $X$ of the hyperfinite $\II_1$ factor $R$.

\begin{thm}
If $W^*(X)$ contains a diffuse element, then $\sigma(X)\geq 1$.
\end{thm}
\begin{proof}
We first note that for any elementary tensor $a\otimes b\in M\otimes M^\circ$, $d(\,\cdot\,):=[\ \cdot\ ,a^*\otimes b^*]$ defines an element of $\mathsf{Der}_{1\otimes 1}(X)$ with
	\[
		d^*(1\otimes 1) = a\tau(b) - \tau(a)b.
	\]
Furthermore,
	\[
		\| ( J_{\tau\otimes \tau^\circ} d(x_1),\ldots, J_{\tau\otimes \tau^\circ} d(x_n) )\|_2^2 \leq 2 \sum_{i=1}^n \|x_i\|^2 \|a\otimes b\|_2^2.
	\]
Thus we can extend the map $a\otimes b\mapsto ( J_{\tau\otimes \tau^\circ} d(x_1),\ldots, J_{\tau\otimes \tau^\circ} d(x_n) )$ into a left $\mmop$-module map
	\[
		\phi\colon L^2(\mmop) \to \overline{\dom(\partial_{X}^*)}.
	\]
If $\phi$ is injective, then it will follow that
	\[
		\sigma(X) = \dim_{\mmop} \overline{\dom(\partial_{X}^*)} \geq \dim_{\mmop} L^2(\mmop)=1.
	\]
Suppose $\eta\in L^2(\mmop)$ satisfies $\phi(\eta)=0$. Consequently, $[x_i,\eta]=0$ for $i=1,\ldots, n$ and so it follows that $[y,\eta]=0$ for all $y\in W^*(X)$. Let $y_0\in W^*(X)$ be a diffuse element, which exists by hypothesis. Since we can identify $L^2(\mmop)\cong \HS(L^2(M))$, $[y_0,\eta]=0$ implies $\eta=0$. Thus $\phi$ is injective.
\end{proof}

\section{Relation to Free Entropy.}
\label{sec:relation_to_free_entropy}

We now turn to an examination of how free Stein irregularity and dimension relate to the free Fisher information and non-microstates free entropy dimension(s).

\begin{thm}\label{thm:fisher_info_and_R-bounded_free_Stein_info}
For $R>0$, $\Sigma^*_R(X\colon B)=0$ if and only if $\Phi^*(X\colon B)\leq R^2$.
\end{thm}
\begin{proof} 
Suppose $\Sigma^*_R(X\colon B)=0$. Then there exists a sequence $(\Xi^{(k)})_{k\in \N}\subset L^2(\BX)$ such that $\|\Xi^{(k)}\|_2\leq R$ and $\Sigma^*(X\mid \Xi^{(k)}\colon B) < \frac1k$ for all $k\in \N$. Let $A_k$ be a free Stein kernel of $X$ relative to $\Xi^{(k)}$ over $B$ such that $\|A_k - \1\|_{\HS} = \Sigma^*(X\mid \Xi^{(k)} \colon B)$. Then $A_k \to \1$. Hence for every $P\in \BT^n$ we have
	\[
		\lim_{k\to\infty} \<\Xi^{(k)}, \ev_X P\>_2 =\lim_{k\to\infty} \<A_{k}, \ev_X\circ\J(P)\>_{\HS} = \< \1, \ev_X\circ\J (P)\>_{\HS}.
	\]
The density of $\BX$ in $L^2(\BX)$ implies the sequence $(\Xi^{(k)})_{k\in \N}$ (since it is uniformly bounded) converges weakly to some $\Xi\in L^2(\BX)^n$. Moreover, the above limit implies $\Xi$ is the conjugate variable of $X$ with respect to $B$ and
	\[
		\Phi^*(X\colon B) = \|\Xi\|_2^2 \leq \liminf_{k\to\infty} \|\Xi^{(k)}\|_2^2 \leq R^2.
	\]
The converse is immediate.
\end{proof}

The following result is a minor generalization of \cite[Theorem 2.7]{Shl04} (which corresponds to the special case $B=\C$). We state it here using our notation and terminology, but the core idea of the proof is not novel.  

\begin{prop}
Let $S$ be a free semicircular family, free from $\BX$. Then
	\[
		\limsup_{t\to 0 } t \Phi^*(X+\sqrt{t}S\colon B) \leq \Sigma^*(X\colon B)^2.
	\]
\end{prop}
\begin{proof}
Let $f\colon (0,\infty)\to (0,\infty)$ be any decreasing function such that
	\begin{align*}
		\lim_{t\to 0} f(t)&=+\infty\text{, but}\\
		\lim_{t\to 0} \sqrt{t} f(t) &=0.
	\end{align*}
(E.g. $f(t)=t^{-1/4})$. Then 
	\[
		\lim_{t\to 0} \Sigma^*_{f(t)}(X\colon B)=\Sigma^*(X\colon B).
	\]
For each $t>0$, let $Q_t\in L^2(\BX)^n$ be such that $\|Q_t\|_2\leq f(t)$ and such that there exists a free Stein kernel $A_t$ for $X$ relative to $Q_t$ over $B$ such that 
	\[
		\| A_t - 1\|_{\HS}\leq \Sigma^*_{f(t)}(X\colon B) + t.
	\]
Recall that the conjugate variables to $X+\sqrt{t}S$  with respect to $B$ are $\mc{E}_t(\frac{1}{\sqrt{t}} S)$ where $\mc{E}_t\colon W^*(B\<X,S\>)\to W^*(B\<X+\sqrt{t} S\>)$ is the conditional expectation (cf. \cite[Corollary 3.9]{VoiV}). By the same proof as in \cite[Lemma 2.3]{Shl04}, it follows that
	\[
		\mc{E}_t(Q_t) = \mc{E}_t\paren{ A_t\# \frac{1}{\sqrt{t}}S}.
	\]
Thus
	\begin{align}\label{eqn:Fisher_Stein_comparison}
		\sqrt{t}\Phi^*(X+\sqrt{t}S\colon B)^{\frac12} &= \sqrt{t} \left\| \mc{E}_t\left(\frac{1}{\sqrt{t}}S\right)\right\|_2\nonumber\\
			&\leq \sqrt{t} \left\| \mc{E}_t\left((\1- A_t)\# \frac{1}{\sqrt{t}} S\right)\right\|_2 + \sqrt{t} \|\mc{E}_t(Q_t)\|_2 \nonumber\\
			&\leq \sqrt{t} \left\|(\1- A_t)\# \frac{1}{\sqrt{t}} S\right\|_2 + \sqrt{t}f(t) \nonumber\\
			& = \| \1- A_t\|_{\HS} + \sqrt{t}f(t) \nonumber\\
			& \leq \Sigma^*_{f(t)}(X\colon B) + t + \sqrt{t}f(t).
	\end{align}
This tends to $\Sigma^*(X\colon B)$ as $t\to 0$.
\end{proof}

We remind the reader that the \emph{relative non-microstates free entropy} of $X$ with respect to $B$ is defined as the quantity
	\[
		\chi^*(X\colon B):= \frac12 \int_0^\infty \frac{n}{1+t} - \Phi^*(X+\sqrt{t} S\colon B)\ dt + \frac{n}{2} \log(2\pi e),
	\]
where $S$ is a free semicircular family free from $\BX$ (cf. \cite[Definition 7.1]{VoiV}). The following is a minor generalization of \cite[Corollary 2.8]{Shl04}. As with the previous result, we state it using our notation and terminology, but the core idea of the proof is not novel.

\begin{prop}\label{prop:perturbed_chi_star_less_than_Stein}
Let $S$ be a free semicircular family free from $X$. Then
	\[
		\limsup_{\epsilon\to0} \frac{\chi^*(X+\sqrt{\epsilon}S\colon B)}{\frac12 \log{\epsilon}} \leq \Sigma^*(X\colon B)^2.
	\]
\end{prop}
\begin{proof}
Using \cite[Corollary 6.14]{VoiV} and implementing the change of variable $t\mapsto t-\epsilon$ in the integral appearing in the above definition of $\chi^*$, we obtain
	\[
	 	\limsup_{\epsilon\to0} \frac{\chi^*(X+\sqrt{\epsilon}S\colon B)}{\frac12 \log{\epsilon}} =\limsup_{\epsilon\to 0} \frac{1}{\log{\epsilon}} \int_\epsilon^1 \frac{n}{1+t-\epsilon} - \Phi^*(X+\sqrt{t}S\colon B)\ dt.
	\]
Now, for any free Stein kernel $A$ relative to some $Q$ over $B$ we have
	\begin{align*}
		\Phi^*(X+\sqrt{t} S\colon B) &\leq \left( \| \mc{E}_t( (\1-A)\#\frac{1}{\sqrt{t}}S)\|_2 + \|\mc{E}_t(A\#\frac{1}{\sqrt{t}} S)\|_2 \right)^2\\
		& \leq \frac1t \|\1 - A\|_{\HS}^2 + \frac{2}{\sqrt{t}} \|\1-A\|_{\HS}\|Q\|_2 + \|Q\|_2^2.
	\end{align*}
Thus
	\begin{align*}
		\int_\epsilon^1 \frac{n}{1+t-\epsilon} - \Phi^*(X+\sqrt{t}S\colon B)\ dt &\geq  n\log(2-\epsilon) + \log(\epsilon)\|\1 - A\|_{\HS}^2\\
																			&\qquad - 4 (1- \sqrt{\epsilon})\|\1-A\|_{\HS}\|Q\|_2 - (1-\epsilon) \|Q\|_2^2.
	\end{align*}
Since $\log(\epsilon)<0$ for $\epsilon<1$, this in turn implies
	\[
		\limsup_{\epsilon\to0} \frac{\chi^*(X+\sqrt{\epsilon}S\colon B)}{\frac12 \log{\epsilon}} \leq \|\1 -A\|_{\HS}^2.
	\]
Since $A$ was an arbitrary free Stein kernel over $B$, we obtain the desired inequality.
\end{proof}

We remind the reader that the there are two versions of the \emph{relative non-microstates free entropy dimension} of $X$ with respect to $B$:
\[
	\delta^*(X\colon B) := n - \liminf_{\epsilon\to 0} \frac{\chi^*(X+\sqrt{\epsilon} S\colon B)}{\frac12 \log{\epsilon}}\qquad\qquad\delta^\star(X\colon B):= n - \liminf_{\epsilon\to 0} t \Phi^*(X+\sqrt{t}S\colon B),
\]
where $S$ is a free semicircular family free from $\BX$; moreover, $\delta^*(X:B) \leq \delta^\star(X:B)$ (cf. \cite[Section 4.1.1]{CS05}\footnote{Although this paper was interested only in the case $B=\C$, the idea generalizes straightforwardly by using the relative versions of $\chi^*$ and $\Phi^*$. }).
Thus from Proposition~\ref{prop:perturbed_chi_star_less_than_Stein} we obtain:

\begin{cor}\label{cor:free_Stein_info_and_free_entropy_dim}
For any $*$-algebra $B$ and $n$-tuple $X$,
	\[
		 \sigma(X\colon B) \leq \delta^*(X\colon B) \leq \delta^\star(X\colon B).
	\]
\end{cor}

	Recall that in the self-adjoint one-variable case $X=(x)$, one has
	\[\delta^*(x) = \delta^\star(x) = 1 - \sum_{t\in\R}\mu(\set{t})^2\]
	by \cite[Proposition 6.3]{VoiII} and \cite[Propositions 7.5 and 7.6]{VoiV}, where $\mu$ is the distribution of $x$ on $\R$. Thus, in particular, the following theorem shows that the above inequalities are in fact equalities.
		
\begin{thm}\label{thm:one_variable_computation}
Let $x\in (M,\tau)$ be self-adjoint with distribution $\mu$ on $\R$. Then
	\[
		\Sigma^*(x)^2 = \sum_{t\in \R} \mu(\{t\})^2.
	\]
Consequently, $\Sigma^*(x)=0$ if and only if $x$ has no atoms.
\end{thm}
\begin{proof}
Recall that in the one-variable case
	\[
		\liminf_{\epsilon\to 0} \frac{\chi^*(X+\sqrt{\epsilon} S\colon B)}{\frac12 \log{\epsilon}} = \sum_{t\in \R} \mu(\{t\})^2.
	\]
Thus Corollary~\ref{cor:free_Stein_info_and_free_entropy_dim} implies
	\[
		\sum_{t\in \R} \mu(\{t\})^2\leq \Sigma^*(x)^2.
	\]
To see the reverse inequality, consider for $\epsilon>0$ the function
	\[
		g_\epsilon(t):= 2\int_\R \frac{(t-s)}{(t-s)^2 +\epsilon^2}\ d\mu(s).
	\]
Observe that $|g_\epsilon(t)| \leq \frac{2}{\epsilon^2}( |t| + \tau(|x|))\in L^2(\mu)$. In particular, for any polynomial $p$ we have
	\begin{align*}
		\int g_\epsilon(t) p(t)\ d\mu(t) &= 2 \iint \frac{(t-s)p(t)}{(t-s)^2 + \epsilon^2}\ d\mu(s)d\mu(t)\\
			&= \iint \frac{(t-s)(p(t) - p(s))}{(t-s)^2 + \epsilon^2}\ d\mu(s)d\mu(t)\\
			&= \iint \frac{(t-s)^2}{(t-s)^2 + \epsilon^2} \frac{p(t) - p(s)}{t-s}\ d\mu(s)d\mu(t).
	\end{align*}
That is, $A_\epsilon(t,s):=\frac{(t-s)^2}{(t-s)^2 + \epsilon^2}$ is a free Stein kernel for $x$ relative to $g_\epsilon$. So we compute for $\delta>0$
	\begin{align*}
		\Sigma^*(x)^2 \leq \|A_\epsilon - \1\|^2_{L^2(\mu)} &= \iint | A_\epsilon(t,s) - \1|^2\ d\mu(t)d\mu(s)\\
			&= \iint \frac{\epsilon^4}{( (t-s)^2 + \epsilon^2)^2}\ d\mu(t)d\mu(s)\\
			&\leq \iint_{|t-s|\geq \delta} \frac{\epsilon^4}{\delta^4}\ d\mu(t)d\mu(s) + \iint_{|t-s|<\delta} 1\ d\mu(t)d\mu(s)\\
			&\leq \frac{\epsilon^4}{\delta^4} + (\mu\otimes \mu)(\{(t,s)\in \R^2\colon |t-s|<\delta\}).
	\end{align*}
Letting first $\epsilon$ tend to zero and then $\delta$, we obtain the other inequality.
\end{proof}

\begin{rem}
As a particular example of Theorem~\ref{thm:one_variable_computation}, for $x\in M_{N}(\C)_{s.a.}$ we have
	\[
		\Sigma^*(x)^2 = \sum_{j=1}^k \frac{m_j^2}{N^2},
	\]
where $k\leq n$ is the number of distinct eigenvalues of $x$ with respective multiplicities $m_1,\ldots, m_k\in \N$.
\end{rem}

The inequalities in Corollary~\ref{cor:free_Stein_info_and_free_entropy_dim} also enable us to prove  the following.

\begin{cor}\label{cor:full_free_Stein_dimension_iff_closable}
Suppose $M=W^*(X)$. Then $\sigma(X)=n$ if and only if $\J$ gives a densely defined closable operator
	\[
		\J\colon L^2(M)^n \to M_n(L^2(\mmop)), 
	\]
and if and only if $\partial$ gives a densely defined closable operator
	\[
		\partial \colon L^2(M)\to L^2(\mmop)^n.	
	\]
\end{cor}
\begin{proof}
	Let us suppose that $\sigma(X) = n$; since $\sigma(X) \leq \delta^*(X) \leq n$, $X$ has full free entropy dimension.
	It then follows from \cite{CS16} that $X$ satisfies no algebraic relation, and hence we may view $\J$ and $\partial$ as densely defined operators with the above domains and codomains.
	Moreover, by Theorem~\ref{thm:free_Stein_info_is_dim}, $\J^*$ and $\partial^*$ are densely defined, whence $\J$ and $\partial$ are closable.

	Contrariwise, when either $\J$ or $\partial$ gives a linear operator, its adjoint as an unbounded operator and its adjoint arising from evaluation of polynomials in $\C\ang{T}$ agree. The closability of $\J$ or $\partial$ is then equivalent to their adjoints having dense domains, and so Theorem~\ref{thm:free_Stein_info_is_dim} yields the result.
\end{proof}

This also allows us to reword Theorem~\ref{thm:extension_to_closure_of_domain} as follows:

\begin{cor}
	Let $M=W^*(X)$. Suppose $\sigma(X) = n$.
	Then for any $Y=(y_1,\ldots,y_m)\in \left(\dom(\bar\partial)\cap M\right)^m$ with $\bar\partial(y_j) \in (\mmop)^n$ for each $j=1,\ldots, m$, we have $\sigma(X,Y) = n$.
\end{cor}

\subsection{Regularity hierarchy}

Let us relate the condition $\sigma(X)=n$ to other well-studied regularity conditions. We have the following picture:
	
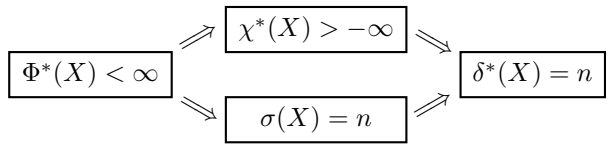
\[
		\begin{tikzpicture}[thick, scale=1.2]

		\node at (0,0) {$\Phi^*(X)<\infty$};
		\draw[thick] (-0.9,-0.25) rectangle (0.9,0.25);

		\node[rotate=30] at (1.2,0.4) {$\implies$};
		\node[rotate=-30] at (1.2,-0.4) {$\implies$};

		\node at (2.5,0.5) {$\chi^*(X)>-\infty$};
		\draw[thick] (1.5,0.25) rectangle (3.5,0.75);

		\node at (2.5,-0.5) {$\sigma(X)=n$};
		\draw[thick] (1.5,-0.25) rectangle (3.5,-0.75);

		\node[rotate=-30] at (3.8,0.35) {$\implies$};
		\node[rotate=30] at (3.8,-0.35) {$\implies$};

		\node at (4.9,0) {$\delta^*(X)=n$};
		\draw[thick] (4.1,-0.25) rectangle (5.7,0.25);

		\end{tikzpicture}
	\]
The top two arrows are of course well-known results: the first is \cite[Proposition 7.9]{VoiV} while the second follows from \cite[Proposition 7.5]{VoiV} and the definition of $\delta^*$ in \cite[Section 4.1.1]{CS05}. The bottom two arrows follow from Theorem \ref{thm:fisher_info_and_R-bounded_free_Stein_info} and Corollary \ref{cor:free_Stein_info_and_free_entropy_dim}, respectively. Thus it is natural to ask what the relationship is between having finite non-microstates free entropy and having full free Stein dimension. In the case $n=1$, we see that the former implies the latter by Theorem~\ref{thm:one_variable_computation}.

\begin{rem}
The above raises some interesting questions:
\begin{enumerate}[label=\arabic*.]
		\item Does $\chi^*(X)>-\infty$ imply $\sigma(X)=n$ in general?
		
		\item Does $\sigma(X)=\delta^*(X)$ in general?
	\end{enumerate}
	We begin to investigate the first question below; then, in Section~\ref{sec:computations}, we exhibit some cases where the equality in the second question holds.
\end{rem}

In order to be begin analyzing the relationship between these two conditions, consider the the following quantity:
	\begin{align}\label{eqn:alpha}
		\alpha:=\limsup_{R\to\infty} \frac{\ln \Sigma_R^*(X)}{\ln R} \in [-\infty, 0].
	\end{align}
That is, $\alpha$ compares how quickly $\Sigma_R^*(X)$ decays as $R$ grows. Note that if $\Sigma^*(X)\neq 0$ we have $\alpha=0$; however, it may be that $\alpha=0$ even when $\Sigma^*(X)=0$. Indeed, consider the Example~\ref{ex:infinite_entropy_full_dimension} below.

\begin{prop}\label{prop:negative_alpha}
With $\alpha$ as above, if $\alpha<0$ then $\chi^*(X)>-\infty$.
\end{prop}
\begin{proof}
Let $\alpha<\beta<0$. Then there exists $R_0>0$ such that for all $R\geq R_0$ we have
	\[
		\Sigma_R^*(X)\leq R^{\beta}.
	\]
Let $\gamma\in (0,1)$. Then substituting $R=\frac{1}{t^{\gamma/2}}$ we have
	\[
		\Sigma_{1/t^{\gamma/2}}^*(X) \leq t^{-\gamma\beta/2}\qquad \forall t<t_0:=\frac{1}{R_0^{2/\gamma}}.
	\]
Using Equation (\ref{eqn:Fisher_Stein_comparison}) we therefore have
	\[
		\Phi^*(X+\sqrt{t} S) \leq \frac1t\left( t^{-\gamma\beta/2} + t^{(1-\gamma)/2}\right)^2 =  \left( t^{(-\gamma\beta - 1)/2} + t^{-\gamma/2}\right)^2 \qquad \forall t< t_0.
	\]
Since $(-\gamma\beta - 1)/2> -1/2$ and $-\gamma/2 > -1/2$ we have that the above quantity is integrable on $[0,t_0]$.
\end{proof}


\section{Some Computations of Free Stein Dimension.}
\label{sec:computations}

We provide some examples in which the free Stein irregularity and dimension can be explicitly computed.
In particular, we show that in these examples the free Stein dimension agrees with the non-microstates free entropy dimensions.
The first result concerns Atiyah's $\ell^2$-Betti numbers for discrete groups (cf. \cite{Ati76, CG86}).
Also see \cite[Chapter 1]{Luc02} for the definition considered here, and \cite{MS05} for the connection to free entropy dimension.

\begin{prop}\label{prop:group_algebras}
Let $\Gamma$ be a discrete group and let $x_1,\ldots, x_n\in \C[\Gamma]_{s.a.}$ generate the group algebra. Then
	\[
		\sigma(x_1,\ldots, x_n) = \beta_1^{(2)}(\Gamma) - \beta_0^{(2)}(\Gamma) + 1,
	\]
where $\beta_0^{(2)}(\Gamma)$ and $\beta_1^{(2)}(\Gamma)$ are the $\ell^2$-Betti numbers of $\Gamma$.
\end{prop}

\begin{proof}
	It was shown in \cite[Theorem 4.1]{MS05} that
	\[
		\delta^*(x_1,\ldots, x_n) = \delta^*(x_1,\ldots, x_n) = \beta_1^{(2)}(\Gamma) - \beta_0^{(2)}(\Gamma) + 1.
	\]
	So, by Corollary \ref{cor:free_Stein_info_and_free_entropy_dim}, it suffices to show
	\[
		\sigma(x_1,\ldots, x_n) \geq \beta_1^{(2)}(\Gamma) - \beta_0^{(2)}(\Gamma) + 1.
	\]
We make use of the following space from \cite[Section 2]{Shl06}: $H_1=\overline{H_1^\circ}^{\HS}$ where
		\begin{align*}
			H_1^\circ:=\text{span}\{ (z_1,\ldots, z_n)\in \HS(L^2(M))^n \colon& \exists Y=Y^* \text{ unbounded, densely defined with}\\
			&1\in \dom(Y),\ [Y,x_j]=z_j\text{ for each } j=1,\ldots,n\}.
		\end{align*}
		We can identify $H_1$ with a closed subspace in $L^2(\mmop)^n$ using the identification
		\begin{align*}
			L^2(\mmop) &\cong \HS(L^2(M))\\
			a\otimes b^\circ &\mapsto aP_1 b,
		\end{align*}
		where $P_1$ is the rank one projection onto $1\in L^2(M)$.
		By \cite[Theorem 1]{Shl06}, for every $Z:=(z_1,\ldots, z_n)\in H_1^\circ$ we have $1\otimes 1\in \dom(\partial_Z^*)$ where $\partial_Z\colon \C\<T\>\to L^2(\mmop)$ is the derivation defined by
		\[
			\partial_Z(p) =\sum_{j=1}^n \ev_X\circ\partial_j(p)\# z_j.
		\]
		Observe that if $J=J_{\tau\otimes\tau^\circ}$ is the Tomita conjugation operator on $L^2(\mmop)$, then for $p\in \C\<X\>$ we have
		\[
			\< 1\otimes 1 ,\ev_X\partial_Z(p)\>_2
			= \sum_{j=1}^n \< 1\otimes 1, \ev_X\partial_j(p) \# z_j\>_2
			= \sum_{j=1}^n \< J z_j, \ev_X\partial_j(p)\>_2
			= \< JZ, \ev_X\partial(p)\>_2.
		\]
		Consequently, $1\otimes 1 \in \dom(\partial_Z^*)$ if and only if $JZ\in \dom(\partial^*)$. It follows that $JH_1\subset \overline{\dom(\partial^*)}$ and so
	\[
		\dim_{\mmop}(\overline{\dom(\partial^*)}) \geq  \dim_{\mmop}(H_1),
	\]
	where the latter dimension is as a right $\mmop$-module. In the proof of \cite[Corollary 4]{Shl06} it was shown that the latter dimension is $\beta_1^{(2)}(\Gamma) - \beta_0^{(2)}(\Gamma)+1$, and so Theorem \ref{thm:free_Stein_info_is_dim} completes the proof.
\end{proof}

Our final example concerns finite-dimensional von Neumann algebras, for which $\delta$, $\delta^\star$, $\delta_0$, and $\Delta$ are known to agree. We show here that $\sigma$ can be added to this list.

\begin{cor}\label{cor:free_Stein_dimension_of_finite-dimensional_algebras}
Consider a finite-dimensional algebra for the form
	\[
		(M,\tau)= \bigoplus_{i=1}^d \left(M_{k_i}(\C), \lambda_i \text{tr}_{k_i}\right),
	\]
where the $\lambda_i$ are positive and sum to one, and $\text{tr}_{k_i}$ is the normalized trace on $M_{k_i}(\C)$. Then for any tuple of generators $X=(x_1,\ldots,x_n)$, we have
	\[
		\sigma(X) =1-\sum_{i=1}^d \frac{\lambda_i^2}{k_i^2}.
	\]
In particular, $\sigma(X)=\delta^*(X)=\delta^*(X)=\delta_0(X) = \Delta(X) = 1-\beta_0(M,\tau)$.
\end{cor}
\begin{proof}
In the proof of \cite[Corollary 5]{Shl06} it is shown that
	\[
		1-\sum_{i=1}^d \frac{\lambda_i^2}{k_i^2} = \delta^*(X)= \dim_{\mmop}(H_1),
	\]	
where $H_1$ is as in the proof of Proposition~\ref{prop:group_algebras}. Hence equality with $\sigma(X)$ follows from the proof of Proposition~\ref{prop:group_algebras}. The remaining equalities are then simply \cite[Corollary 5]{Shl06} (see also \cite{CS05}, namely Proposition 2.9 and Equation 3.10).
\end{proof}

\titleformat{\section}[hang]
{\normalfont\Large\bfseries}
{Appendix \thesection.}{0.5em}{}
\appendix
\section{}
\label{appendix:maikernel}
In this appendix we will demonstrate that for $X$ self-adjoint and algebraically free, the Mai kernel $A_\Xi$ (given in Proposition~\ref{prop:Mai}) satisfies
	\[
		\|A_\Xi - \1\|_{\HS} = \Sigma^*(X\mid \Xi)
	\]
if and only if $\Xi=0$. We emphasize that any free Stein kernel attaining the free Stein discrepancy of $X$ is necessarily contained in the closure of the range of $\J$.

Let $d : L^2(M) \to L^2(\mmop)$ be the derivation given by commutation against $1\otimes1$: $\zeta \mapsto \zeta\otimes 1 - 1\otimes \zeta$.
Given $Z = (\zeta_1, \ldots, \zeta_n) \in L^2(M)^n$, let $D : L^2(M)^n \to L^2(\mmop)^n$ be given by applying $d$ to each coordinate: $D(Z) = (d(\zeta_1), \ldots, d(\zeta_n))$.
\begin{lem}
	\label{lem:dotdx}
	Suppose that $A\in \overline{\partial\C\ang{X}} \subseteq L^2(\mmop)^n$.
	If $(p_k)_{k\in\N}$ is a sequence in $\C\ang{X}$ so that
	\[
		A = \lim_{k\to\infty}\partial p_k,
	\]
	then 
	\[
		A\cdot D(X) = \lim_{k\to\infty}(p_k\otimes 1 - 1\otimes p_k).
	\]
\end{lem}
\begin{proof}
	Observe that $D(X)\in (\mmop)^n$, so that it has a bounded right action on $L^2(\mmop)^n$. Thus the equation follows from a straightforward computation:
	\[
		A\cdot D(X)
		= \lim_{k\to\infty} (\partial p_k) \cdot D(X)\\
		= \lim_{k\to\infty}\sum_{j=1}^n \partial_j p_k \# (x_j\otimes 1 - 1\otimes x_j) \\
		= \lim_{k\to\infty} p_k\otimes1 - 1\otimes p_k.
		\qedhere
	\]
\end{proof}

The lemma applies, in particular, to the rows of any free Stein kernel that attains the free Stein discrepancy of $X$.

\begin{prop}
	Suppose $\Xi = (\xi_1, \ldots, \xi_n) \in L^2(M)^n \ominus \C^n$, and let $A_\Xi$ be as in Proposition~\ref{prop:Mai}:
	\[
		A_\Xi:=\left[\frac12 (\xi_i\otimes 1 - 1\otimes \xi_i)\#(x_j\otimes 1 - 1\otimes x_j^\circ) \right]_{i,j=1}^n\in M_n(L^2(\mmop)).
	\]
	If $\|A_\Xi - \1\|_{\HS}= \Sigma^*(X \mid \Xi)$, then $\Xi = 0$.
\end{prop}

\begin{proof}
	First note that it suffices to assume that $\tau(x_1)=\cdots=\tau(x_n)=0$. Indeed, let
	\[
		\mathring{X}=(\mathring{x_1},\ldots, \mathring{x_n}) = (x_1 - \tau(x_1), \ldots, x_n - \tau(x_n)).
	\]
Then clearly $\C\langle\mathring{X}\rangle = \C\<X\>$ and consequently $\Xi\in L^2(W^*(\mathring{X}))^n$. Moreover, $A_\Xi$ is unchanged when replacing $X$ with $\mathring{X}$. Now for any
	\[
		A \in M_n(L^2( W^*(X)\bar\otimes W^*(X)^\circ)) = M_n(L^2( W^*(\mathring{X})\bar\otimes W^*(\mathring{X})^\circ)),
	\]
if $A$ is a free Stein kernel for $X$ relative to $\Xi$, then by the chain rule it is also a free Stein kernel for $\mathring{X}$ relative to $\Xi$, and vice versa. Hence $\Sigma^*(X\mid \Xi) = \Sigma^*(\mathring{X}\mid \Xi)$ and so, replacing $X$ with $\mathring{X}$ if necessary, we may assume $\tau(x_1)=\cdots =\tau(x_n)=0$.

	Note that the $i$-th row of $A_\Xi$ is given by $\frac12 d(\xi_i) \# D(X) =: r_i$, and from the assumption that $\Sigma^*(X \mid \Xi) = \norm{A_\Xi - \1}_{\HS}$ we have that $r_i \in \overline{\partial \C\<X\>}\subset L^2(\mmop)^n$.
	Now, pick $(p_k)_{k\in\N}$ in $\C\ang{X}$ so that $\partial p_k \to r_i$; since $\C1 \in \ker\partial$, we may assume $\tau(p_k) = 0$, replacing $p_k$ by $p_k - \tau(p_k)$ if needed.
	Then from Lemma~\ref{lem:dotdx}, we have $\displaystyle r_i\cdot D(X) = \lim_{k\to\infty} p_k\otimes1-1\otimes p_k$.
	Hence
		\[
			(1\otimes\tau^\circ)(r_i\cdot D(X)) = \lim_{k\to\infty} p_k \qquad \text{and} \qquad (\tau\otimes1)(r_i\cdot D(X)) = - \lim_{k\to\infty}p_k.
		\]
	We compute
	\begin{align*}
		\sum_{j=1}^n &\xi_ix_j^2\otimes1 - 2\xi_i x_j\otimes x_j + \xi_i\otimes x_j^2 - x_j^2 \otimes \xi_i + 2 x_j\otimes x_j\xi_i - 1\otimes x_j^2\xi_i \\
		&= 2r_i\cdot D(X) \\
		&= 2\lim_{k\to\infty}p_k\otimes1 - 1\otimes p_k \\
		&= [(1\otimes\tau^\circ)(2r_i\cdot D(X))] \otimes 1 + 1\otimes[ (\tau\otimes1)(2r_i\cdot D(X)) ] \\
		&= \paren{ \sum_{j=1}^n \xi_i (x_j^2 +\tau(x_j^2)) +2x_j \tau(x_j \xi_i) - \tau(x_j^2\xi_i)}\otimes 1 + 1\otimes \paren{ \sum_{j=1}^n \tau(\xi_i x_j^2) - 2 \tau(\xi_i x_j) x_j - (\tau(x_j^2)+x_j^2) \xi_i}\\
		&= \sum_{j=1}^n \xi_i x_j^2 \otimes 1 + \tau(x_j^2) d(\xi_i) + 2\tau(x_j\xi_i)d(x_j) - 1\otimes x_j^2 \xi_i.
	\end{align*}
	Subtracting common terms on each side, we find
	\begin{align}
		\label{eqn:somestuffiszero}
		\sum_{j=1}^n \xi_i\otimes [x_j^2 - \tau(x_j)^2] + 2 x_j\otimes [x_j \xi_i - \tau(x_j \xi_i)] - 2 [\xi_i x_j - \tau(\xi_i x_j)]\otimes x_j - [x_j^2 - \tau(x_j^2)]\otimes \xi_i=0.
	\end{align}
	As $X$ is algebraically free, we may find polynomials $p$ and $q$ such that $\ang{x_1^2, p} = 1$ while $p$ is orthogonal to all other monomials of degree at most two, and $\ang{x_1, q} = 1$ while $q$ is orthogonal to all other monomials of degree at most three.
	Applying the map $1\otimes\ang{\ \cdot\,, p}_2$ to the above equality yields
	\[
		\xi_i + \sum_{j=1}^n 2 x_j \ang{ x_j \xi_i, p} - [x_j^2 - \tau(x_j^2)]\ang{\xi_i, p}=0,
	\]
	whence $\xi_i$ is a polynomial in $X$ of degree at most two.
	Now, applying $1\otimes \ang{\ \cdot\,, q}_2$ to Equation~\ref{eqn:somestuffiszero} and using the fact that $x_j\xi_i$ is a polynomial of degree at most three, we find
	\[
		2x_1\ang{x_1\xi_i, q} - 2(\xi_ix_1 - \tau(\xi_ix_1)) - \ang{\xi_i, q}\sum_{j=1}^n (x_j^2-\tau(x_j^2)) = 0.
	\]
	From this it follows that $\xi_ix_1$ is a linear combination of $1, x_1, x_1^2, x_2^2, \ldots, x_n^2$.
	But then $\xi_i$ must be a linear combination of $1$ and $x_1$; say $\xi_i = s + tx_1$.
	Looking at the coefficient of $x_1^2$ in the above equation, we find that $-2t -\ang{\xi_i, q} = 0$; since $\ang{\xi_i, q} = t$, we have $t = 0$, whence $\xi_i \in \C$.
	As $\xi_i \in L^2(M)\ominus \C$, $\xi_i = 0$.
\end{proof}

\section{}
\label{app:examples}
	In this appendix we consider a few informative examples. 
	The first two show that for certain tuples generating interpolated free group factors $L(\mathbb{F}_t)$, the parameter $t$ can be recovered through a formula involving the free Stein dimension of the tuples.

\begin{ex}
Let $s_0,s_1,\ldots, s_n$ be a free semicircular family. Let $B=W^*(s_0)$ and for each $j=1,\ldots, n$, let $e_j,f_j$ be projections in $B$ that are either equal or orthogonal. Define $k_j=1$ if $e_j=f_j$ and $k_j=2$ otherwise. Then by \cite{Rad94} we have
	\[
		M:=W^*(s_0,e_1s_1f_1,\ldots, e_ns_nf_n)\cong L(\mathbb{F}_t)
	\]
where
	\[
		t= 1+ \sum_{s=1}^n k_s \tau(e_s)\tau(f_s).
	\]

Let $K=\sum_{j=1}^n k_j$ and let $X$ be the $K$-tuple consisting of the $e_js_jf_j$ (and $f_js_je_j$ if $e_j\neq f_j$). We claim
	\begin{align}\label{eqn:FSD_of_Radulescu_LFt}
		\sigma(X\colon B) +\sigma(s_0)=t.
	\end{align}
Indeed, define $P_t$ to be the $K\times K$ diagonal matrix whose $(i,i)$ entry is $e_j\otimes f_j$ if $x_i=e_js_jf_j$. Note that $P_t$ is a projection, and by freeness one easily sees that $P_t\in\dom(\J_{X\colon B}^*)$ with $\J_{X\colon B}^*(P_t) = X$. Consequently, $\Sigma^*(X\colon B) \leq \|P_t -\1\|_{\HS}$. On the other hand, $\ev_X T = X = P_t\# X = \ev_X(P_t\# T)$. Thus for any $A\in \dom(\J^*_{X\colon B})$ we have
	\[
		\<A,\1\>_{\HS} = \<\J_{X\colon B}^*(A), \ev_X T\>_2 = \< \J_{X\colon B}^*, \ev_X(P_t\# T)\>_2 = \<A, P_t\>_{\HS}.
	\]
Consequently
	\begin{align*}
		\|A-\1\|_{\HS}^2 - \|P_t - \1\|_{\HS}^2 &= \|A\|_{\HS}^2 - 2\Re \<A,\1\>_{\HS}  - \|P_t\|_{\HS}^2 + 2\Re\<P_t,\1\>_{\HS}\\
			&=\|A\|_{\HS}^2  - 2\Re\<A,P_t\>_{\HS} + \|P_t\|_{\HS}^2  = \|A - P_t\|_{\HS}^2 \geq 0.
	\end{align*}
Thus
	\[
		\Sigma^*(X\colon B)^2 = \|P_t - \1\|_{\HS}^2 = \sum_{j=1}^n k_2\tau\otimes\tau^\circ(1\otimes 1 - e_j\otimes f_j) = K - \sum_{j=1}^n k_2 \tau(e_j)\tau(f_j) = K+1-t.
	\]
Equation~(\ref{eqn:FSD_of_Radulescu_LFt}) then follows since $\sigma(s_0)=1$. $\hfill\blacksquare$
\end{ex}

\begin{ex}
	Fix a finite, connected graph $\Gamma= (V,E)$  with vertex weighting $\mu\colon V\to [0,1]$ satisfying $\sum_{v\in V} \mu(v)=1$, and let $\vec{\Gamma}=(V,\vec{E})$ be the associated directed graph (cf. \cite{HN18}). Recall that the free graph von Neumann algebra $(\mc{M}(\Gamma,\mu),\tau)$ is generated by operators $\{x_\epsilon\colon \epsilon\in \vec{E}\}$ and an orthogonal family of projections $\{p_v\colon v\in V\}$, which satisfy the following graph relations:
	\begin{itemize}
		\item $\tau(p_v) = \mu(v)$ for all $v\in V$;

		\item $x_{\epsilon}^* = x_{\epsilon^\text{op}}$ for all $\epsilon\in \vec{E}$;

		\item $p_v x_{\epsilon} p_w = \delta_{v=s(\epsilon)} \delta_{w=t(\epsilon)} x_{\epsilon}$ for all $v,w\in V$ and $\epsilon \in \vec{E}$.
	\end{itemize}
Moreover, there is a trace-preserving isomorphism between $M$ and the interpolated free group factor $L(\mathbb{F}_t)$ with parameter
	\[
		t:=1-\sum_{v\in V} \mu(v)^2 + \sum_{v\in V} \mu(v) \sum_{w\sim v} n_{v,w}\mu(w),
	\]	
where $n_{v,w}$ is the number of edges connecting $v$ to $w$.

	Let $X:=(x_{\epsilon}\colon \epsilon\in \vec{E})$, $Y:=(p_v\colon v\in V)$, and $B:=\C\<Y\>$. We claim
		\begin{align}\label{eqn:FSD_of_free_graph_algebra}
			\sigma(X\colon B) + \sigma(Y) = t.
		\end{align}
By \cite[Lemma 3.9]{Har17} (see also \cite[Lemma 2.1]{HN18}), one has $P_V\in \dom(\J_{X\colon B}^*)$ where $P_V$ is the projection given by the $|\vec{E}|\times |\vec{E}|$ diagonal matrix with $(\epsilon,\epsilon)$-entry given by $p_{s(\epsilon)}\otimes p_{t(\epsilon)}$. Then one has $\Sigma^*(X\colon B) \leq \|P_V - \1\|_\HS$. On the other hand, observe that $\ev_X T = X=P_V\# X= \ev_X (P_V\# T)$. So the other inequality follows by precisely the same argument as in the previous example. Thus
	\begin{align*}
		\Sigma^*(X\colon B)^2 = \|P_V - \1\|_{\HS}^2 &= \sum_{\epsilon \in \vec{E}}  \tau\otimes \tau^\circ( 1\otimes 1 - p_{s(\epsilon)}\otimes p_{t(\epsilon)})\\
		&= \sum_{\epsilon \in \vec{E}}  1 - \mu(s(\epsilon)) \mu(t(\epsilon))\\
		&= |\vec{E}| - \sum_{v\in V} \mu(v) \sum_{w\sim v} n_{v,w} \mu(w), 
	\end{align*}	
Finally, appealing to Corollary~\ref{cor:free_Stein_dimension_of_finite-dimensional_algebras} yields Equation~(\ref{eqn:FSD_of_free_graph_algebra}). $\hfill\blacksquare$
\end{ex}

	The next example was concocted to demonstrate explicitly that $\alpha=0$ in Equation~(\ref{eqn:alpha}) does not imply $\Sigma^*(x)>0$.
	It also demonstrates the fact that full free entropy dimension is strictly weaker than finite free entropy, by explicitly constructing a probability measure with no atoms and infinite logarithmic energy; while this result is already known, we are not aware of an explicit example in the literature.

	\begin{ex}\label{ex:infinite_entropy_full_dimension}
		Let $I_n \subset [0, 1]$ be a disjoint sequence of intervals such that the Lebesgue measure $\lambda(I_n) < e^{-12^n}$.
		Define a function $f$ as follows:
		\begin{align*}
			f : \R &\to \R_{\geq0} \\
			t &\mapsto \sum_{n=1}^\infty \frac{1}{2^n\lambda(I_n)}\chi_{I_n}(t).
		\end{align*}
		By construction $f$ is non-negative, integrable, and has mass 1, so it is a probability density; let $\mu$ be the measure with density given by $f$.
		We claim that the (negative) logarithmic energy of $\mu$ is infinite.
		Indeed,
		\begin{align*}
			\iint_{\R^2}\log\abs{x-y}\,d\mu(x)\,d\mu(y)
			&\leq \sum_{n=1}^\infty \iint_{I_n^2} \log\abs{x-y}\,d\mu(x)\,d\mu(y) \\
			&\leq \sum_{n=1}^\infty \log(e^{-12^n})4^{-n} \\
			&= -\infty.
		\end{align*}
		Now, since $\supp(\mu)$ is bounded and has a diffuse component, there exists a bounded, self-adjoint, algebraically free operator $x$ with spectral measure $\mu$. It follows from Proposition~\ref{prop:negative_alpha} that $\alpha=0$, and by Theorem~\ref{thm:one_variable_computation} we have $\Sigma^*(x)=0$. $\hfill\blacksquare$
	\end{ex}

As a decreasing convex function, if $R \mapsto \Sigma^*_R(X)$ ever plateaus it remains constant forever.
This happens, for example, when conjugate variables actually exist: $\Sigma^*_R(X) = 0$ for $R \geq \sqrt{\Phi^*(X)}$.
One may wonder, then, if this behaviour can occur when $\Sigma^*(X)>0$; we provide a family of examples to show that it can.

\begin{ex}
	Let $\mu = \frac12\sigma + \frac12\delta_a$ where $d\sigma=\chi_{[-2,2]}(t)\frac{1}{2\pi}\sqrt{4-t^2}\ dt$ is the semicircle law.
	Then we will show that if $|a| > 2$ and $x$ has spectral measure $\mu$, then there is $R > 0$ so that $\Sigma^*_R(x) = \Sigma^*(x) = \frac12$.

	As in the proof of Theorem~\ref{thm:one_variable_computation}, define the functions
	$$g_\epsilon(t) := 2\int_\R \frac{t-s}{(t-s)^2+\epsilon^2}\,d\mu(s).$$
	As before, we have a free Stein kernel for $x$ relative to $g_\epsilon$ given by
	$$A_\epsilon(t, s) := \frac{(t-s)^2}{(t-s)^2+\epsilon^2}.$$
	Notice that as $\epsilon \to 0$, $A_\epsilon(t, s) \to \chi_{t\neq s} =: A(s, t)$ which has $\|A - \1\|_{L^2(\mu\times \mu)}^2 = \mu(\{a\})^2$; so it suffices to show that $A$ is a free Stein kernel.

	Here we will use the fact that $a \notin \supp(\sigma)$ to conclude that $g_\epsilon$ converges in $L^2(\mu)$ as $\epsilon \to 0$.
	This can be checked by, for example, recognizing that $g_\epsilon$ converges in both $L^2(\delta_a)$ and $L^2(\sigma)$: in the former space, 
	$$g_\epsilon \to \int \frac{1}{a-s}\,d\sigma(s),$$
	which converges since $a$ is outside the support of $\sigma$;
	in the latter,
	$$g_\epsilon(t) \to Kt + \frac{1}{t-a},$$
	where we have used the fact that the Hilbert transform of the semicircle distribution is $t$ while $a$ is, once again, outside of the support of $\sigma$.
	Let $g = \displaystyle\lim_{\epsilon\to0}g_\epsilon$ with the limit in $L^2(\mu)$.

	We claim that $A$, above, is a free Stein kernel for $x$ relative to $g$, whereupon $\Sigma_R^*(x) = \mu(\{a\}) = \frac12$ for $R \geq \|g\|_{L^2(\mu)}$. (However, note that $\|g\|_{L^2(\mu)} $ diverges as $|a|\to 2$.)
	To see that, notice that $\partial^*$ is closed since $\partial$ is densely defined. Since $A_\epsilon \in \dom(\partial^*)$ with $\partial^*(A_\epsilon)=g_\epsilon$ (by virtue of being a free Stein kernel) we therefore have $A\in \dom(\partial^*)$ with $\partial^*(A)=g$. That is, $A$ is a free Stein kernel for $x$ relative to $g$.
	$\hfill\blacksquare$
\end{ex}

	One may be tempted to guess that if $A$ is a Stein kernel for $X$ and $Y$ is arbitrary that there is some Stein kernel for $(X, Y)$ of the form $$\begin{pmatrix}A&*\\ *&*\end{pmatrix}.$$
	This is true when $Y \in \C\ang{X}^m$ or when $Y$ is free from $X$.
	However, this does not happen in general.

	\begin{ex}
		\label{ex:no_extension}
		Let $\mu$ be any measure which is diffuse and so that $\frac{1}{s+t} \notin L^2([0,1]^2, \mu\times\mu)$.
		Note that $s^2$ is diffuse as $s$ is and so $1 \in \overline{\dom\paren{\partial_{s^2}^*}}$; we will show that there is no element of the form $(1,\alpha)\in \overline{\dom(\partial_{(s^2,s)}^*)}$.

		Notice that because $(0, s)$ and $(s^2, s)$ generate the same algebra, the map $\rho_{(0, s), (s^2, s)}$ from Proposition~\ref{prop:change_of_variables} provides a bijection between the closures of the free Stein kernels.
		In particular, if $(0, b) \in \overline{\dom(\partial_{(0,s)}^*)}$ then $((s+t)b, b) \in \overline{\dom(\partial_{(s^2,s)}^*)}$, and \emph{every element is of this form.}
		Hence if $(1, \alpha)$ were to be in the domain of $\partial_{(s^2, s)}^*$, we would have $\alpha = \frac{1}{s+t}$, which is absurd, as we would then have $\alpha \notin L^2([0,1]^2, \mu\times\mu)$.
	$\hfill\blacksquare$
	\end{ex}


\bibliographystyle{amsalpha}
\bibliography{references}

\end{document}